\documentclass[letterpaper,10pt]{article} 
\usepackage{osameet3} %% use version 3 for proper copyright statement

\usepackage[utf8]{inputenc}
\usepackage[english]{babel}
\usepackage[nottoc]{tocbibind}
\usepackage{lmodern}

\usepackage{mathtools}
\usepackage{latexsym}
\usepackage{shortvrb}
\usepackage{fancyhdr}
\usepackage{amsfonts}
\usepackage{amsmath}
\usepackage{amssymb}
\usepackage{mathrsfs}
\usepackage{amsthm}
\usepackage{cite}
\usepackage{enumitem}
\usepackage{import}
\usepackage{tabularx}
\usepackage{comment}

\usepackage{hyperref}
\hypersetup{
	colorlinks=true,
	linkcolor=blue,
	filecolor=magenta,      
	urlcolor=cyan,
}
\usepackage{geometry}

\usepackage{pgf,tikz,pgfplots}
\pgfplotsset{compat=1.17}
\usepackage{mathrsfs}
\usetikzlibrary{arrows}
\usetikzlibrary{matrix}
\usetikzlibrary{arrows}

\newtheorem{thm}{Theorem}[section]
\newtheorem{lem}[thm]{Lemma}
\newtheorem{prop}[thm]{Proposition}
\newtheorem{ex}[thm]{Example}
\newtheorem{cor}[thm]{Corollary}
\newtheorem{notation}[thm]{Notation}

\theoremstyle{definition}
\newtheorem{definition}[thm]{Definition}
\newtheorem{obs}[thm]{Observation}

\DeclareMathOperator{\Set}{\mathbf{Set}}

\DeclareMathOperator{\catd}{\mathcal{D}}

\DeclareMathOperator{\catG}{\mathcal{G}}

\newcommand\authormark[1]{\textsuperscript{#1}}

\usepackage{amsmath,amssymb}
\usepackage{comment}

\begin{document}

%Newly defined operators

\newcommand{\rtarr}{\longrightarrow}
\newcommand{\ltarr}{\longleftarrow}
\newcommand{\monoto}{\hookrightarrow}
\newcommand{\epito}{\twoheadrightarrow}
\newcommand{\fto}{\xrightarrow}
\newcommand{\fot}{\xleftarrow}
\newcommand{\ot}{\leftarrow}
\newcommand{\impl}{\Rightarrow}
\newcommand{\toeq}{\stackrel{\simeq}{\to}}
\newcommand{\mh}{\mbox{-}} % math hyphen
\newcommand{\cn}{\colon}
\newcommand{\id}{\mathrm{id}}
\newcommand{\Id}{\mathrm{Id}}
\newcommand{\ob}{\mathrm{ob}}
\newcommand{\ev}{\mathrm{ev}}
\newcommand{\op}{\mathrm{\textrm{op}}}
\newcommand{\e}{\mathrm{e}}
\newcommand{\al}{\alpha}
\newcommand{\be}{\beta}
\newcommand{\ga}{\gamma}
\newcommand{\de}{\delta}
\newcommand{\pa}{\partial}   %pretend its Greek
\newcommand{\epz}{\varepsilon}
\newcommand{\ph}{\phi}
\newcommand{\phy}{\varphi}
\newcommand{\phz}{\varphi}
\newcommand{\et}{\eta}
\newcommand{\io}{\iota}
\newcommand{\ka}{\kappa}
\newcommand{\la}{\lambda}
\newcommand{\tha}{\theta}
\newcommand{\thz}{\vartheta}
\newcommand{\si}{\sigma}
\newcommand{\ze}{\zeta}
\newcommand{\om}{\omega}
\newcommand{\Ga}{\Gamma}
\newcommand{\GA}{\Gamma}
\newcommand{\La}{\Lambda}
\newcommand{\LA}{\Lambda}
\newcommand{\De}{\Delta}
\newcommand{\DE}{\Delta}
\newcommand{\Si}{\Sigma}
\newcommand{\SI}{\Sigma}
\newcommand{\Th}{\Theta}
\newcommand{\Om}{\Omega}
\newcommand{\OM}{\Omega}
\newcommand{\Ups}{\Upsilon}
\newcommand{\UPS}{\Upsilon}
\newcommand{\cc}{\mathbb{C}}
\newcommand{\wt}{\widetilde}
\newcommand{\strictslice}[2]{#1\hspace{-2pt}\sslash_{#2}^{\textrm{str}}}
\newcommand{\strongslice}[2]{#1\hspace{-2pt}\sslash_{#2}}
\newcommand{\conslice}[2]{#1\hspace{-2pt}\sslash_{#2}^{\textrm{adj}}}
\newcommand{\opconslice}[2]{#1\hspace{-2pt}\sslash_{#2}^{\textrm{opadj}}}
\renewcommand{\cal}{\mathcal}
\newcommand{\Simp}{\textbf{SimpSet}}
\newcommand{\omegagpd}{\omega\textbf{Gpd}}
\newcommand{\Gpd}{\textbf{Gpd}}
\newcommand{\dblGpd}{\textbf{dGpd}^!}
\newcommand{\Grp}{\textbf{Grp}}
\newcommand{\inftygrpd}{\infty\Gpd}
\newcommand{\FGp}{\textbf{F}\Grp}
\newcommand{\CubeSet}{\textbf{CubeSet}}
\newcommand{\Homset}{\mbox{Hom}}
\newcommand{\rhogpdX}{\rho(|X|)}
\newcommand{\rhogpdsigmaX}{\rho(|\Sigma^*X|)}
\newcommand{\rhogpdglobX}{\rho^\bigcirc(|X|)}
\newcommand{\rhogpdglobsigmaX}{\rho^\bigcirc(|\Sigma^*X|)}
\newcommand{\lo}{\mbox{local LHGF}}
\newcommand{\model}{\sigma(X,G)}
\newcommand{\fieldsXG}{\mathcal{M}(X,G)}
\newcommand{\fieldsYG}{\mathcal{M}(Y,G)}
\newcommand{\fieldsXK}{\mathcal{M}(X,K)}
\newcommand{\fieldsYK}{\mathcal{M}(Y,K)}
\newcommand{\fieldsXoneK}{\mathcal{M}(X_1,K)}
\newcommand{\fieldsXtwoK}{\mathcal{M}(X_2,K)}
\newcommand{\fieldsXGglob}{\mathcal{M}^\bigcirc(X,G)}
\newcommand{\fieldsXKglob}{\mathcal{M}^\bigcirc(X,K)}
\newcommand{\fieldsXKcube}{[P^1\rhogpdX,BK]}
\newcommand{\fieldsXGcube}{[P^1\rhogpdX,BG]}

\title{Higher homotopy and lattice gauge fields}

% \author{Author name(s)}
% \address{Author affiliation and full address}
% \email{e-mail address}
%%Uncomment the following line to override copyright year from the default current year.
%%Uncomment the following line to override copyright year from the default current year.

\author{Juan Orendain,\authormark{1} José Antonio Zapata,\authormark{2}}

\address{\authormark{1} Case Western Reserve University \\
\authormark{2}Centro de Ciencias Matemáticas, Universidad Nacional Autónoma de México campus Morelia.\\}

\email{\authormark{1}juan.orendain@case.edu \\ \authormark{2} zapata@matmor.unam.mx} %% email address is required

\begin{abstract}
\begin{comment}
     We present a general formalism for higher dimensional versions of lattice gauge fields based on higher strict homotopy groupoids. We make use of machinery in nonabelian Algebraic Topology to define local lattice higher gauge fields. We provide local-to-global principles for lattice higher gauge fields based on the HHSvK theorem, we study sets of lattice higher gauge fields geometrically, and we provide a higher dimensional notion of gauge equivalence. We prove that, under the correct assumptions, lattice higher gauge field as presented here, subsumes both the notion of extended lattice gauge field and other notions of higher lattice gauge field present in the literature.
\end{comment}
 We present a general formalism for higher dimensional versions of lattice gauge fields based on higher strict homotopy groupoids. 
 First, using the language of nonabelian Algebraic Topology, we define local lattice higher gauge fields. Then, we provide local-to-global principles for lattice higher gauge fields based on the HHSvK theorem. We prove that, under the correct assumptions, lattice higher gauge field as presented here, subsumes both the notion of extended lattice gauge field and other notions of higher lattice gauge field present in the literature.
\end{abstract}

\section{Introduction}

\noindent Gauge theory studies connections on principal bundles. A connection $A$ on a principal bundle $\pi:E\to M$, with structure Lie group $G$, represents the basic field, while the group $G$ represents the gauge group of the system. In locally trivial coordinates, the data of a gauge field on a principal bundle $\pi$ is equivalent to the data provided by the $A$ induced parallel transport 
along paths in a neighborhood $U$ in $M$. Precisely, if $\mathcal{P}_1(U)$ represents the path groupoid of $U$, i.e. the groupoid of smooth thin equivalence classes of smooth lazy paths in $U$, and $BG$ denotes the delooping groupoid of $G$, then the data of a connection $A$ on $\pi$ is equivalent, in locally trivial coordinates, to a smooth functor Hol$_A:\mathcal{P}_1(U)\to BG$, see \cite{BaezSchreiber,BaezHuerta,Pfeiffer}. The correspondence between connection 1-forms and parallel transport functors is referred to as the correspondence between the differential and the integral formulation of gauge theory in \cite{PfeifferGirelli}. Higher gauge theory studies parallel transport 
along higher dimensional paths. The notion of higher dimensional path can mean different things, and in this context the notion of gauge group of a higher dimensional gauge field can also mean different things. In \cite{BaezSchreiber,BaezHuerta} 2-dimensional paths are taken to mean thin homotopy equivalences of smooth maps, %restricting to given paths on the two lazy paths on the hemispheres, 
and gauge groups are taken to be Lie 2-groups. In this form 2-dimensional paths organize into a smooth 2-groupoid $\mathcal{P}_2(U)$, and 2-dimensional parallel transport is described as a strict 2-functor Hol$_2:\mathcal{P}_2(U)\to\mathcal{G}$ from the 2-path groupoid $\mathcal{P}_2(U)$ to a gauge Lie 2-group $\mathcal{G}$.

Lattice gauge fields determine parallel transport along paths in discretized versions of a manifold $M$
\footnote{In applications to particle physics, computer simulations usually model discretized spacetime as a cubulation of the 4-torus. Here we follow the convention 
of modelling discretized spacetime by a triangulation of $M$.}. 
Precisely, lattice gauge fields on a triangulation $X$ of of a manifold $M$, with gauge group $G$, describe the parallel transport of paths on $M$ that only traverse the 1-skeleton of $X$. Following the ideas of \cite{BaezSchreiber,Barret}, this should take the form of a functor $A$ from a groupoid of paths on $X$ to the delooping groupoid $BG$ of $G$. See \cite{Pfeiffer} for a formal presentation of these ideas, where discrete paths are to be considered as compatible words of edges of $X$ and their formal inverses, i.e where the path groupoid of $X$ is taken to be the free gropoid generated by the 1-skeleton of $X$. Higher lattice gauge fields describe parallel transport along higher dimensional paths on triangulations $X$ of the base manifold $M$ of a principal bundle. As is the case with smooth higher gauge fields, a discrete higher dimensional path can mean several things, as well as the gauge group of a theory. Lattice gauge fields in dimension 2 have been described in \cite{Pfeiffer} as strict 2-functors from the free 2-groupoid generated by the 2-skeleton of the triangulation $X$ to a gauge Lie 2-group $\mathcal{G}$, and thus 2-dimensional paths are respresented as pasting diagrams of simplexes in $X$ and their formal inverses, and 2-dimensional lattice gauge fields thus label these paths with the structure given by $\mathcal{G}$, in a way compatible with pasting. This treatment models the behaviour of a gauge field on a local trivialization after discretizing the corresponding patch of spacetime.

Extended lattice gauge fields, introduced in \cite{MenesesZapata1,MenesesZapata2}, are a type of discretized gauge field which extends usual lattice gauge field. Apart from characterizing parallel transport along paths fitting in the 1-skeleton of $X$, in \cite{MenesesZapata1} it is proven that the information provided by an extended lattice gauge field is enough to recover a principal $G$-bundle up to bundle equivalence over $|X|\approx M$. 
An extended lattice gauge field can be extracted from a gauge field in the continuum. Apart from extracting information about the parallel transport along links of the 1-skeleton of $X$, the field records the homotopy class of homotopies of parallel transport maps corresponding to certain homotopies of paths called simplexes of paths and associated to pairs of nested simplexes in $X$. This extra information determines a Cech cocycle. 
Thus, extended lattice gauge fields are a discretization of gauge fields capturing global features of the gauge field. 
In this sense, extended lattice gauge fields generalize standard lattice gauge fields differently than the higher gauge fields \cite{Pfeiffer} described in the previous paragraph, which are local in nature.

In this paper we provide a uniform treatment of local higher lattice gauge fields, subsuming the usual formulation of lattice gauge theory described above and higher dimensional lattice gauge fields as described in \cite{Pfeiffer}. 
Our fundamental object of study will be that of a local higher lattice gauge field, or $\lo$ for short. Our concept of $\lo$ is based on the theory of groupoids internal to cubical $\omega$-groupoids of \cite{BrownBook} and formally describes parallel transport along discrete cubical surfaces of arbitrary dimensions. The main ingredient of our formalism will be the homotopy cubical $\omega$-groupoid $\rho(|X|)$ of the geometric realization of a simplicial set $X$ playing the role of discretized spacetime, and thus higher dimensional paths will be modelled as thin homotopy classes of singular cubes lying on the corresponding skeleton of $X$. 
Our choice of cubical shapes for higher dimensional paths has been made both for taste and convenience, but the fact that the category of cubical $\omega$-groupoids is equivalent to the category of $\infty$-groupoids allows us to rephrase the theory in terms of globular higher paths. Our definitions and results will usually be supplemented with their equivalent globular version. 
We define global higher lattice gauge fields as compatible collections of local higher lattice gauge fields. 
We address the local-to-global problem and give a succinct solution for globular LHGFs making use of the Higher Homotopy Seifert Van Kampen Theorem of \cite{BrownBook}. 
Further, we provide a brief description of how globular global LHGFs define a notion of higher parallel transport. \textit{Note:} In our notion of $\lo$ we will follow the notation conventions of \cite{BrownBook,Pfeiffer,BaezSchreiber} when composing higher dimensional paths, i.e., we compose paths from left to right, as opposed to composing paths from right to left, as in \cite{KobayashiNomizu,BaezHuerta,MenesesZapata1,MenesesZapata2}. This will lead to parallel transport associated to global LHGFs to act on torsors on the right, as opposed to acting on the left, which in turn means that the base space of the associated fibre bundle will correspond to the orbits of a global left action, as opposed to the orbits of a right action. This can be remedied by considering our definitions on opposite internal $\omega$-groupoids. We have made the choice not to do this so our notation agrees with the notation of \cite{BrownBook} in this work, but we reserve the right to make this adjustment in future work.

\textbf{Plan of the paper:} In Section \ref{sec:Intro} we review the theory of cubical $\omega$-groupoids and $\infty$-groupoids, in Section \ref{s:Internal} we set the technical background for the rest of the paper by developing, to some extent, the theory of groups and groupoids internal to the category of cubical $\omega$-groupoids. In Section \ref{s:LHGF} we make use of the theory developed in Section \ref{s:Internal} to present our definition of $\lo$. We study $\lo$'s in terms of labeling of higher dimensional surfaces, we prove that when restricted to simplicial sets of dimensions 1 and 2, $\lo$s coincide with higher lattice gauge fields as in \cite{Pfeiffer}, and we provide sets of $\lo$s with a topology. Our definition of $\lo$ assumes contractibility of the corresponding piece of discretized spacetime. In Section \ref{s:GaugeEquivalence} we present a higher dimensional version of local gauge transformation, generalizing that of \cite{MenesesZapata2}, and we prove that local gauge transformations organize the set of $\lo$s into a groupoid. In section \ref{s:LocaltoGlobal} we define global LHGFs as compatible collections of local LHGFs. Here we drop the assumption that discretized spacetime should be contractible and instead consider good covers. Then we address the local to global problem. 
In Section \ref{s:ELGF} we prove that under the right geometric conditions on the simplicial set $X$ we can associate, to every global globular LHGF an ELGF as in \cite{MenesesZapata1,MenesesZapata2}; 
as a consequence, we prove that a global globular LHGF determines a principal $G$-bundle unique up to bundle equivalence. Finally, in Section \ref{s:parallel} we briefly explain how a globular global LHGF defines a notion of higher parallel transport along homtopy globes of paths.

\section{Cubical $\omega$-groupoids}\label{sec:Intro}
In this first section we recall the definition of cubical set, cubical set with connections, cubical $\omega$-groupoid, and we recall the construction of the homotopy $\omega$-groupoid associated to a filtered topological space. We will mostly follow the presentation and notational conventions in \cite{Algebraofcubes,BrownGlobular,GrandisMauri}. 
\subsection{Cubical Sets}\label{ss:cubicalsets}
The restricted cubical site $\mathbb{I}$ is the subcategory of \textbf{Set} whose objects are the elementary cubes $2^k,k\geq 0$ and whose morphisms are functions $f:2^k\to 2^m$ deleting entries and inserting 0's and 1's, not modifying the order of the remaining entries. If we identify the elementary cube $2^k$ with the integer trace of the standard cube $[0,1]^k$, then $\mathbb{I}$ is generated by the integer restriction of the maps $D_i^\alpha$ from $[0,1]^k$ to $[0,1]^{k+1}$ and $E_i$ from $[0,1]^{k+1}$ to $[0,1]^k$ defined by the formulas
\begin{equation}\label{Dialpha}
D_i^\alpha(v_1,...,v_k)=(v_1,...,v_{i-1},\alpha,v_i,...,v_k) ,   
\end{equation}
for $\alpha\in \left\{0,1\right\}$ and $i\in\left\{1,...,k+1\right\}$, and
\begin{equation}\label{Ei}
    E_i(v_1,....,v_{k+1})=(v_1,...,v_{i-1},v_{i+1},...,v_{k+1}),
\end{equation}
for $i\in \left\{1,...,k+1\right\}$. The maps $D_i^\alpha$ and $E_i$ satisfy the opposite cubical relations, forming a presentation of $\mathbb{I}$ \cite{GrandisMauri}. The intermediate cubical site $\mathbb{J}$ is the subcategory of \textbf{Set} generated by $\mathbb{I}$ and connection maps, i.e. $\mathbb{J}$ is the category generated by $\mathbb{I}$ together with the maps $G_i:[0,1]^{k+1}\to [0,1]^k$ defined by the equation:
\begin{equation}\label{Gi}
    G_i(v_1,...,v_{k+1})=(v_1,...,max\left\{v_i,v_{i+1}\right\},...,v_{k+1}),
\end{equation}
for $i\in\left\{1,...,k+1\right\}$. The maps $G_i$ satisfy the opposite connection relations, relating them to the maps $D_i,E_i$, and generating $\mathbb{J}$, see \cite{GrandisMauri}. 

A cubical set is a presheaf on the reduced cubical site $\mathbb{I}$. Concretely, a cubical set $K_*$ is a sequence of sets $K_k$, with $k\geq 0$, maps $\partial_i^\alpha:K_k\to K_{k-1}$ for $i\in\left\{0,...,k-1\right\}$ and $\alpha\in\left\{0,1\right\}$, maps $\epsilon_i:K_{k-1}\to K_k$ for $i\in\left\{0,...,k-1\right\}$, all satisfying the the cubical relations, see \cite{Algebraofcubes}. The maps $\epsilon_i$ and $\partial_i^\alpha$ are called the degeneracies and face maps of $K_*$. When it is clear from context that we speak of cubical sets we will drop the $\ast$ from the above notation, i.e. we will write $K$ and not $K_*$ for a cubical set. A cubical set with connections is an extension of a cubical set to a presheaf on the intermediate cubical site $\mathbb{J}$. Concretely, a cubical set with connections is a cubical set $K$ as above, together with additional degeneracy maps $\Gamma_i:K_{k-1}\to K_k$ for $i\in\left\{0,...,k-1\right\}$ satisfying the connection relations, c.f. \cite{Algebraofcubes}. The maps $\Gamma_i$ are called the connections of $K$. Connections on cubical sets allow us to speak of commutative higher dimensional cubes. The main example of a cubical set with connections is the following: Let $X$ be a topological space. The singular cubical set $KX$ associated to $X$ is the representable presheaf Hom$_{\mbox{\textbf{Top}}}(\_,X)$. Concretely, $K_kX$ is the set of all singular cubes on $X$, i.e. $K_kX$ is the set of all continuous maps from $[0,1]^k$ to $X$. Face, degeneracy and connection maps for $K_kX$ are defined by pulling back the corresponding generating morphisms of $\mathbb{J}$, i.e. by precomposing singular cubes with the generating maps $D_i^\alpha,E_i,G_i$ of $\mathbb{J}$. 
\subsection{Cubical $\omega$-groupoids}\label{ss:omega groupoids}
An $\omega$-groupoid is a cubical set with connections $K$ together with $k$ partially defined binary operations $+_i:K_k\times_{\partial_i^1,\partial_i^0}K_k\to K_k$ and $k$ 1-ary operations $-_i:K_k\to K_k$ for every $k\geq 1$ and $i\in\left\{0,...,k-1\right\}$, providing $K_k$ with $k$ groupoid structures, having $\epsilon_i$ as identities, each of these satisfying compatibility conditions with the cubical set structure on $K$, and all satisfying the exchange relations c.f. \cite{Algebraofcubes}.

The singular cubical set $KX$ of a topological space $X$ is naturally provided with composition and inversion defined by concatenation operations $+_i$ and the operations $-_i$ of retracing cubes in opposite coordinate directions. In order to make the operations $+_i$ and $-_i$ satisfy the axioms defining a cubical $\omega$-groupoid we must consider a quotient of $K_*X$ by an appropriate notion of filtered homotopy equivalence.

A filtration $X_*$ on a topological space $X$ is a sequence of inclusions $X_0\subseteq X_1\subseteq ...\subseteq X$ such that $X=\bigcup_{i=1}^\infty X_i$. A filtered space is a topological space $X$ together with a filtration $X_*$. We identify a filtered topological space with its filtration $X_*$. Every CW-complex $C$ is a filtered space through its skeletal filtration, i.e. $C_k$ is the $k$-dimensional skeleton of $C$. The standard $n$-dimensional cube $[0,1]^n$ is thus a filtered space $[0,1]^n_*$ with its skeletal filtration. We write $[0,1]^\infty$ for $\bigcup_{n=1}^\infty [0,1]^m$. The skeletal filtrations on $[0,1]^m$ thus define a filtration $[0,1]^\infty_*$ on $[0,1]^\infty$. A filtered map $X_*\to Y_*$ between filtered spaces $X_*$ and $Y_*$ is a sequence of continuous maps $f_k:X_k\to Y_k$ such that $f_{k+1}|_{X_k}= f_k$ for every $k$. Filtered spaces and filtered maps organize into a category \textbf{FTop}. 

Let $X$ be a filtered space. The singular cubical set $KX$ associated to $X$ is the presheaf  Hom$_{\mbox{\textbf{FTop}}}(\_,X)$. The set of $k$-dimensional cubes of $KX$ is thus the set of filtered maps $f:[0,1]^m_*\to X_*$. Given a filtered space $X_*$ and singular cubes $f$ and $g$ on $X$, \textbf{filtered homotopy rel vertices} from $f$ to $g$, is a continuous map $H:[0,1]^\infty\times [0,1]\to Y$ such that for every $k$: $H([0,1]^k\times [0,1])\subseteq X_k$, $H|_{[0,1]^0\times [0,1]}$ is constant, $H|_{[0,1]^k\times\left\{0\right\}}=f_k$, and $H|_{[0,1]^k\times\left\{1\right\}}=g_k$. Of particular interest in this paper is filtered homotopy of singular cubes on the geometric realization $|X|$ of a simplicial set, filtered by its skeletal filtration. Filtered homotopy rel vertices in that case is usually referred to as thin homotopy rel vertices.

Filtered homotopy rel vertices defines an equivalence relation $\equiv$ on $KX$. Write $\rho(X)$ for $RX/\equiv$ and write $\langle f\rangle$ for the equivalence class of a singular filtered cube $f$ on $X$ modulo $\equiv$. We call equivalence classes $\langle f\rangle$ of singular cubes $f$ on $X$ modulo $\equiv$, homotopy cubes. Face maps, degeneracies and connections on $RX$ induce face maps, degeneracies, and connections on $\rho(X)$. Given homotopy cubes $\langle f\rangle$ and $\langle g\rangle$ in $\rho(X)_m$ such that $\partial_i^1 f\equiv \partial_i^0 g$, the sum $\langle f\rangle +_i \langle g\rangle$ is the homotopy cube $\langle f +_i\nu+_i g\rangle$ where $\nu$ is an $m$-dimensional cube in $KX$ such that $\partial_i^0\nu=\partial_i^1 f$, $\partial_i^1\nu=\partial_i^0 g$ and such that $\nu$ is degenerate. These operations are well defined and provide $\rho(X)$ with the structure of an $\omega$-groupoid, see \cite{BrownBook}. The $\omega$-groupoid $\rho(X)$ is the \textbf{homotopy $\omega$-groupoid} of the filtered space $X$. 

The category of cubical sets \textbf{CubeSet} is the category of presheaves $Psh(\mathbb{I})$ on the restricted cubical site $\mathbb{I}$, i.e. the objects of \textbf{CubeSet} are cubical sets, and morphisms are natural transformations. Concretely, given cubical sets $K$ and $K'$, morphisms, in \textbf{CubeSet}, from $K$ to $K'$, are sequences $A$ of functions $A_k:K_k\to K'_k$ satisfying the equations: $\partial_i^{'\alpha} A_k=A_{k-1}\partial_i^\alpha$ and $\epsilon'_iA_{k-1}=A_k\epsilon_i$ for every $k\geq 1$, $i\in\left\{0,...,k-1\right\}$ and $\alpha\in\left\{0,1\right\}$. Likewise, the category of cubical sets with connections \textbf{CubeSet}$^*$ is the category of presheaves $Psh(\mathbb{J})$ on the intermediate cubical site $\mathbb{J}$, i.e. objects of \textbf{CubeSet}$^*$ are cubical sets with connections, and morphisms are natural transformations. $\omega$-groupoids are groupoids internal to \textbf{CubeSet}$^*$. Internal morphisms thus organize $\omega$-groupoids with connections into a category $\omega$\textbf{Gpd}. Given two $\omega$-groupoids with connections $K$ and $K'$, morphisms, in $\omega$\textbf{Gpd}, from $K$ to $K'$, are morphisms $A$, in \textbf{CubeSet}$^*$, intertwining the groupoid operations $+_i$ and $-_i$ of $K$ and $K'$. 

\subsection{$\infty$-groupoids}\label{ss:infinity}
The elementary $n$-globe $G^n$ is the $n$-dimensional disk $D^n$ together with the cell decomposition
\begin{equation}\label{eq:globulardec}
D^n=e_1^{\pm 1}\cup ...\cup e_n^{\pm 1}\cup e_n
\end{equation}
where $e_i^{\pm 1}$ is the set of points $x=(x_1,...,x_n)\in D^n$ such that $||x||=1$, $x_j=0$ for every $j<n-i$ and $x_{n-i}>0$. The globular site $\mathbb{G}$ is the category whose objects are the elementary $n$-globes $G^n$ and whose morphisms are generated by maps $\overline{d}_i^{\pm 1}:G^i\to G^n$ embedding $G^i$ in $G^n$ along the decomposition (\ref{eq:globulardec}), and projection maps $\overline{s}_i:G^n\to G^i$ such that $\overline{s}_i(x_1,...,x_n)=(x_1,...,x_i)$. The maps $\overline{d}_i^{\pm 1}$ and $\overline{s}_i$ satisfy the opposite globular relations, see \cite{BrownGlobular}. 

A globular set is a presheaf on $\mathbb{G}$. Concretely, a globular set is a sequence $G_*$ of sets $G_k$, with $k\geq 0$, and maps $d_i^{\pm 1}:G_n\to G_i$ and $s_i:G_i\to G_n$ for every $n\geq 0$ and every $i\in\left\{1,...,n\right\}$, satisfying the globular relations. A cube $a$ in a cubical set $K$ is a \textbf{globular cube} if $\partial_i^{\alpha}a$ is in the image of $\epsilon_i^i$ for every $i$. The collection $HK$ of globular cubes of a cubical set $K$ inherits, from the cubical structure of $K$, the structure of a globular set. 

An $\infty$-groupoid is a globular set $G$ together with partially defined operations $+_i:G_n\times_{d_i^+,d_i^-}G_n\to G_n$ and $-_i:G_n\to G_n$ for every $n\geq 0$ and $i\in\left\{0,...,n-1\right\}$, providing $G_n$ with $n$-groupoid structures, each having $s_i$ as identity, and satisfying compatibility conditions with the globular set structure on $G$ and the exchange relations. $\infty$-groupoids organize into a category $\infty$\textbf{Gpd} in a way similar to $\omega \textbf{Gpd}$.
 
 Given an $\omega$-groupoid $K$, the globular set $H K$ associated to $K$ inherits, from $K$, the structure of $\infty$-groupoid. In particular, every filtered space is associated a homotopy $\infty$-groupoid $\rho^\bigcirc(X)$, the $\infty$-groupoid $H \rho(X)$ associated to the homotopy $\omega$-groupoid $\rho(X)$. The correspondence $K\mapsto H K$ extends to an equivalence $H:\omega\mbox{\textbf{Gpd}}\to\infty\mbox{\textbf{Gpd}}$, see \cite{BrownGlobular}.

%newsection
\section{Internal groups and groupoids}\label{s:Internal}
\noindent In this section we study groups and groupoids internal to $\omega$-groupoids, and their internal morphisms.
\subsection{$\omega$-groupoids of paths}\label{ss:paths}
The category \textbf{CubeSet} is symmetric monoidal with tensor product operation $K\otimes K'$ defined by the formula $(K\otimes K')_n=\bigsqcup_{m+k=n}K_m\times K_k$. Boundary and degeneracy maps are defined analogously. Connections and $\omega$-groupoid structures are compatible with the above monoidal structure, and thus \textbf{CubeSet}$^*$ and $\omega$\textbf{Gpd} inherit from \textbf{CubeSet}, symmetric monoidal structures. 
The above monoidal structure on $\omega$\textbf{Gpd} is closed, with internal hom object $[K,K']$ given by the formula
\begin{equation}\label{eq:InternalHom}
[K,K']_m=\mbox{Hom}_{\omega\mbox{\textbf{Gpd}}}(P^mK,K')
\end{equation}
\noindent where $P^mK$ is the left path $\omega$-groupoid associated to $K$, i.e. $P^{m}K_k=K_{k+m}$, $\hat{\partial_i^\alpha}:P^mK_k\to P^mK_{k-1}$ is the $m+i,\alpha$-th face function $\partial^\alpha_{m+i}$ of $K_{m+k}$, and $\hat{\epsilon_i}$, $\hat{+}_i$ are defined similarly, c.f. \cite{BrownBook}. Boundary, degeneracy, connection, composition, and inverse maps in the internal Hom $\omega$-groupoid $[K,K']$, are defined by pushing forward the corresponding leftover maps of $P^mK$ along morphisms. In the rest of the paper we will omit the $\hat{}$ part of the operations defined on $P^1K$ above.
\subsection{Groupoids internal to $\omega$-groupoids}\label{ss:Internalgroupoids}
\noindent We consider groupoids strictly internal to $\omega$\textbf{Gpd} (considered as a Cartesian symmetric monoidal category, as opposed to the monoidal structure defined in the previous subsection). It will be useful to record the data of groupoid internal to $\omega$\textbf{Gpd} in certain detail. An internal groupoid $K^\bullet$ in $\omega$\textbf{Gpd} is a pair $(K^1,K^0)$ of $\omega$-groupoids, together with source and target morphims $s,t:K^1\to K^0$, unit morphism $u:K^0\to K^1$, composition morphism $\odot:K^1\times^{t,s}_{K^0}K^1\to K^1$ and inverse morphism $(\_)^{-1}:K^1\to K^1$, satisfying the usual conditions source, target, unit, composition, and inversion operations on a groupoid satisfy, on the nose. We call $K^0$ and $K^1$ the $\omega$-groupoid of objects, and the $\omega$-groupoid of morphisms, of $K^\bullet$ respectively. We will usually omit the $\bullet$ part of the above notation, and simply write $K$ for a groupoid internal to $\omega\Gpd$. We will also write $K=(K^1,K^0)$. The following Lemma provides our main example of groupoids internal to $\omega$\textbf{Gpd}.

\begin{lem}\label{lem:PathsInternalGroupoid}
Let $K\in\omegagpd$. The pair $(P^1K,K)$ inherits, from $P^1K$ and $K$ the structure of a groupoid internal to $\omega$\textbf{Gpd}. 
\end{lem}
\begin{proof}
Let $K$ be an $\omega$-groupoid. The sequences formed by the maps $\partial_0^\alpha:P^1K_m\to K_m$, $\epsilon_0:K_m\to P^1K_m$, $+_0:P^1K_m\times_{K_m}P^1K:m\to P^1K_m$, and $-_0:P^1K_m\to P^1K_m$ defined for every $m\geq 0$ assemble into morphisms $\partial_0^\alpha:P^1K\to K$, $\epsilon_0:K\to P^1K$, and $+_0:P^1K\times_{K} P^1K\to P^1K$ and $-_0:P^1K\to P^1K$. The fact that these morphisms define on $(P^1K,K)$ the structure of a groupoid internal to $\omega$\textbf{Gpd} follows directly from the fact that both $P^1K$ and $K$ are $\omega$-groupoids.
\end{proof}

\noindent We will denote the internal groupoid defined in Lemma \ref{lem:PathsInternalGroupoid} $P^1K$, or $P^1K^\bullet$ when differentiation with $\omega$-groupoid $P^1K$ is needed. Groupoids internal to $\omega$\textbf{Gpd} organize into a category in the obvious way: Given two internal groupoids $K=(K^1,K^0)$ and $L =(L^1,L^0)$, an internal morphism from $K$ to $L$ is a pair $A=(A^1,A^0)$ where $A^1:K^1\to L^1$ and $A^0:K^0\to L^0$ are morphism of $\omega$-groupoids satisfying the obvious compatibility relations between the internal groupoid structures of $K$ and $L$. The map associating to every $\omega$-groupoid $K$, the path $\omega$-groupoid $P^1K$ is an endofunctor $P^1:\omega\mbox{\textbf{Gpd}}\to\omega\mbox{\textbf{Gpd}}$. The endofunctor $P^1$ extends to a functor $P^{1\bullet}:\omega\mbox{\textbf{Gpd}}\to \textbf{\mbox{Gpd}}_{\omega\mbox{\textbf{Gpd}}}$ associating to every $\omega$-groupoid $K$ the internal groupoid $P^1K^\bullet$ of Lemma \ref{lem:PathsInternalGroupoid}.

\begin{prop}\label{prop:P1InternalGroupoid}
    Let $K^\bullet =(K^1,K^0)$ be a groupoid internal to $\omega$\textbf{Gpd}. The pair $P^1K^\bullet=(P^1K^1,P^1K^0)$ inherits, from $K$, the structure of a groupoid internal to $\omega$\textbf{Gpd}. 
\end{prop}
\noindent We will call $P^1K^\bullet$ the path internal groupoid associated to the internal groupoid $K$. We can iterate the construction of $P^1K^\bullet$ to the construction of higher path internal groupoids $P^kK^\bullet$ in the obvious way.

\begin{prop}\label{prop:EnrichedInternalGroupoids}
Let $K, L\in\Gpd_{\omega\Gpd}$. There is an $\omega$-groupoid $[K,L]$ where 
\[[K,L]_k=Hom_{\Gpd_{\omega\Gpd}}(P^kK^\bullet,K^\bullet)\]
\noindent for every $k\geq 0$, making \textbf{Gpd}$_{\omega\mbox{\textbf{Gpd}}}$ into a $\omega$\textbf{Gpd}-enriched category.
\end{prop}

\subsection{Groups internal to $\omega$-groupoids.}\label{ss:Internalgroups}
We also consider groups internal to $\omega$\textbf{Gpd}. As we did in subsection \ref{ss:Internalgroupoids}, we record the definition of group internal to $\omega$\textbf{Gpd} in some detail. A group internal to $\omega$\textbf{Gpd} is an $\omega$-groupoid $K$ together with a binary operation morphism $\odot:K\times K\to K$, a 1-ary operation morphism $\_^{-1}:K\to K$, and a 0-ary operation morphism $1:\mbox{\textbf{0}}\to K$, satisfying the usual axioms of a group, on the nose. A morphism of internal groups between internal groups $K$ and $L$ is a morphism of $\omega$-groupoids $A:K\to L$ intertwining the corresponding internal operations on $K$ and $L$. Internal groups and their morphisms organize into a category \textbf{Gp}$_{\omega\mbox{\textbf{Gpd}}}$. The cartesian monoidal category structure on $\omega$\textbf{Gpd} extends to a cartesian monoidal structure on \textbf{Gp}$_{\omega\mbox{\textbf{Gpd}}}$ in an obvious way.

A filtered topological group is a topological group $G$ together with a filtration $G_*$ of $G$ as a topological space, such that $G_i$ is a subgroup of $G_{i+1}$ for every $i$. As we do with filtered topological spaces, we identify a filtered topological group with its underlying group $G$. Every topological group $G$ has a trivial filtered group $G_*$ associated to it, where $G_i=G$ for every $i$. Filtered topological groups are precisely groups internal to \textbf{FTop} and thus organize into a cartesian monoidal category, which we write as \textbf{Gp}$_{\mbox{\textbf{FTop}}}$. Given a filtered topological group $G$, the group operations of $G$ extend to filtered maps $\cdot:G\times G\to G$, $\_^{-1}:G\to G$ and $1:\mbox{\textbf{0}}\to K$ on $G$. Applying the functor $\rho:\mbox{\textbf{FTop}}\to\omega\mbox{\textbf{Gpd}}$ to the filtered topological space underlying $G$ and the filtered maps $\cdot:G\times G\to G$, $\_^{-1}:G\to G$ and $1:\mbox{\textbf{0}}\to K$ we obtain an $\omega$-groupoid $\rho(G)$ together with morphisms of $\omega$-groupoids $\odot:\rho(G)\times\rho(G)\to\rho(G)$, $\_^{-1}:\rho(G)\to\rho(G)$ and $1:\mbox{\textbf{0}}\to \rho(G)$, satisfying the conditions defining a group. We obtain the following lemma.

\begin{lem}\label{Lem:RhoGInternalGroupoid}
Let $G$ be a filtered group. The group operations on $G$ provide $\rho(G)$ with the structure of a group internal to $\omega$\textbf{Gpd}. Moreover, the assignment $G\mapsto \rho(G)$ extends to a functor from $\mbox{\textbf{Gp}}_{\mbox{\textbf{Ftop}}}$ to $\omega$\textbf{Gpd}.
\end{lem}

\noindent Lemma \ref{Lem:RhoGInternalGroupoid} provides a way of generating groups internal to $\omega$\textbf{Gpd} from filtered topological groups. The following Lemma says that the functor $\rho$ is isomorphism-dense on internal groups, i.e. every group internal to $\omega$\textbf{Gpd} can be presented, up to isomorphisms, as an internal groupoid as in the statement of Lemma \ref{Lem:RhoGInternalGroupoid}.
\begin{lem}\label{lem:InternalGroups}
Let $K$ be a group internal to $\omega$\textbf{Gpd}. There exists a filtered group $G$ such that $K$ is isomorphic to $\rho(G)$.
\end{lem}

\begin{proof}
Let $K$ be a group internal to $\omega$\textbf{Gpd}. Let $sk^\ast K$ be the skeletal filtration of $K$. Then $K\cong \rho(|sk K|)$, see \cite[Proposition 14.5.3]{BrownBook}. The set $sk^mK$ is the set $\bigcup_{k\leq m}K_k$. The operations $\odot_m$, $\_^{-1}_m$ and $1_m$ make $sk^mK_*$ into a group. Moreover $sk^mK$ is a subgroup of $ sk^{m+1}K$ for every $m$. The group operations $\odot_m$, $\_^{-1}_m$ and $1_m$ on $sk^mK_*$ extend to corresponding continuous operations $\odot_m$, $\_^{-1}_m$ and $1_m$ on $|sk^mK|$, providing $|sk^mK|$ with the structure of a topological group. Moreover, since $sk^mK$ is a subgroup of $ sk^{m+1}K$ then $|sk^mK|$ is a subgroup of $ |sk^{m+1}K|$ for every $m$, and thus $|skK|$ is a filtered topological group. Observe that the isomorphism $K\cong \rho(|skK|)$ preserves the internal group structure and thus is an isomorphism of internal groupoids.
\end{proof}
\noindent In the following Proposition, we write $\textbf{0}$ for the terminal $\omega$-groupoid.

\begin{prop}\label{prop:Delooping}
 Let $K$ be a group internal to $\omega$\textbf{Gpd}. The internal group operations on $K$ provide the pair $BK=(K,\mbox{\textbf{0}})$  with the structure of a groupoid internal to $\omega$-\textbf{Gpd}. Moreover, the map $K\mapsto BK$ extends to an embedding $B:\mbox{\textbf{Gp}}_{\omega\mbox{\textbf{Gpd}}}\to\mbox{\textbf{Gpd}}_{\omega\mbox{\textbf{Gpd}}}$
\end{prop} 
\begin{proof}
Let $\textbf{1},\odot,(\_)^{-1}$ be the structure of group internal to $\omega$\textbf{Gpd} of $K$. Make $s,t:K\to\mbox{\textbf{0}}$ be the unique morphism from $K$ to the terminal $\omega$-groupoid $\textbf{0}$. The tuple $s,t,\textbf{1},\odot,(\_)^{-1}$ provides $BK^\bullet$ with the structure of a groupoid internal to $\omega$\textbf{Gpd}, which follows from the fact that $K$ is an internal group. The fact that the assignment $K\mapsto BK$ extends to a full and faithful functor is immediate from the definition of $BK$. 
\end{proof}

\subsection{Groupoids internal to $\infty$\textbf{Gpd}}\label{ss:globularinternal}
\noindent We now consider groupoids internal to $\infty$-groupoids. The structure defining groupoids internal to $\infty$-groupoids is similar to the structure defining groupoids internal to $\omega$-groupoids. A groupoid internal to $\infty$\textbf{Gpd} is an $\infty$-groupoid, together with morphisms $\odot:K\times_{K_0} K\to K$, $(\_)^{-1}:K\to K$, and $1\in K$, satisfying the usual axioms of a group. Groupoids internal to $\infty$\textbf{Gpd} organize into a category \textbf{Gpd}$_{\infty\mbox{\textbf{Gpd}}}$. Groups internal to $\infty$\textbf{Gpd}. We denoe the category of groups internal to $\infty$\textbf{Gpd} by \textbf{Gp}$_{\infty\textbf{Gpd}}$.

Let $K$ be an $\infty$-groupoid. Write $K[-1]$ for the sequence of sets $K[-1]_n=K_{n-1}$ for every $k\geq 1$ and $K[-1]_0=\left\{\bullet\right\}$ a singleton set. The $\infty$-groupoid structure shifts to maps $d[-1]_i^{\pm 1}=d_{i-1}^{\pm 1}$, $s[-1]_i=s_{i-1}$, $+[-1]_i=+_{i-1}$, $(\_)[-1]_*^{-1}=(\_)^{-1}_*$ on the shifted sets $K[-1]_n$. These maps are not enough to induce a structure of globular set, nor of $\infty$-groupoid on $K[-1]$. The following proposition says that if we assume an internal group structure on $K$, the shifted maps defined above can be completed to an $\infty$-groupoid structure on $K[-1]$.
\begin{prop}\label{prop:Shift}
    Let $K$ be a group internal to $\infty$-groupoids. The internal group structure on $K$ defines, on $K[-1]$, the structure of an $\infty$-groupoid.
\end{prop}
\begin{proof}
    Let $\odot:K\times K\to K$, $(\_)^{-1}:K\to K$ be the morphisms of $\infty$-groupoids providing $K$ with the structure of group internal to $\infty$-groupoids. Let $n\geq 0$. Let $d_0^{\pm}:K_n[-1]\to K_0$ be the constant maps on the single element $\bullet$ of $K[-1]_0$. Let $s_0:K[-1]_0\to K[-1]_n$ be the constant function on $1_n$. Observe that $K[-1]_n\times_{d_0^+,d_0^-} K[-1]_n=K[-1]_n\times K[-1]_n$, and make $+[-1]_0:K[-1]_n\times K[-1]_n\to K[-1]_n$ be the internal product morphism $\odot_{n-1}$, and finally, make $(\_)[-1]^{-1}_0$ be the internal inversion morphism $(\_)^{-1}_{n-1}$. It is easily seen that with $d[-1]^{\pm 1}_0,s[-1]_0,+[-1]_0, (\_)[-1]^{-1}_0$ thus defined, the collection of maps of $\infty$-groupoids $d[-1]^{\pm 1},s[-1],+[-1], (\_)[-1]^{-1}$, with $n\geq 0$, and now $i\in\left\{0,...,n-1\right\}$, provides the sequence $K[-1]$ with the structure of $\infty$-groupoid.
\end{proof}

\noindent Given $K\in\omegagpd$ and $L\in\Grp_{\omegagpd}$, we will be interested in internal morphisms $A: P^1K^\bullet\to BL$. We can associate to the information encoded in any such morphism, a sequence of maps $\hat{A}$ from the squence of sets underlying the $\infty$-groupoid $H K$ to the sequence of sets underlying the $\infty$-groupoid $H L[-1]$ by setting $\hat{A}_0$ be the constant function on the unique 0-dimensional cube of $H L[-1]$, and by setting the $\hat{A}_n=A^1_{n-1}$ for every $n\geq 1$. The following proposition says that the correspondence $A\mapsto \hat{A}$ defines a morphism of $\infty$-groupoids and is natural with respect to both $K$ and $L$.

\begin{prop}\label{prop:equivalenceGlobularCubic}
    Let $K\in\omegagpd$ and $L\in\Grp_{\omegagpd}$. The correspondence $A\mapsto \hat{A}$ described above defines a natural map:

    \[\hat{(\_)}:\mbox{Hom}_{\mbox{\textbf{Gpd}}_{\omega\mbox{\textbf{Gpd}}}}(P^1K,BL)\to \mbox{Hom}_{\infty\mbox{\textbf{Gpd}}}(H K,H L[-1])\]
\end{prop}
\begin{proof}
    Let $A: P^1K\to BL$ be a morphism in \textbf{Gpd}$_{\omega}$\mbox{\textbf{Gpd}}. The sequence of maps $\hat{A}:H K\to H L[-1]$ is compatible with the structure maps $d^{\pm 1}_i,s_i,+_i, -_i$ of $H K$ and the shifted structure maps $d[-1]^{\pm 1}_i,s[-1]_i,+[-1]_i, -[-1]_i$ of $H L$ with $i\geq 1$ follows from the fact that $A^1$ is a morphism of $\omega$-groupoids from $P^1K$ to $L$. Compatibility between the structure maps $d^{\pm 1}_0,s_0,+_0, -_0$ of $H K$ and the shifted structure maps $d[-1]^{\pm 1}_0,s[-1]_0,+[-1]_0, -[-1]_0$ of $H L$ follows from compatibility of $A$ with the internal group structures of $P^1K$ and $BL$. The sequence $\hat{A}$ thus defines a morphism of $\infty$-groupoids from $H K$ to $H L[-1]$. Naturality of the assignment $\hat{(\_)}$ is immediate.
\end{proof}
%Question: Is this an equivalence??

\section{Lattice higher gauge fields}\label{s:LHGF}
In this section we use the concepts developed in the previous two sections to provide our promised model of lattice higher gauge fields. 
\textit{Notes:}  (i) $|X|$ is endowed with its skeletal fibration. Filtered homotopy in the case of skeletal fibrations is called {\em thin homotopy}, and this is the relevant notion of homotopy among paths needed to define gauge fields. 
(ii) Throughout this section, and until Section \ref{s:LocaltoGlobal}, we will assume that simplicial sets $X$ are part of a good cover of a larger simplicial set, and thus we will implicitly assume that the geometric realization $|X|$ of any simplicial set considered, is contractible. In our context, the larger simplicial set may be a discretization of a manifold of physical interest, e.g. spacetime. 
(iii) An essential ingredient of our construction is an internal group, and the family of internal groups $B\rho(G)$ associated to a filtered topological group $G$  will play an important role below.
Of particular interest in physics and differential geometry are the cases where $G$ is is a Lie group with a trivial filtration. 

\subsection{Local LHGFs}\label{ss:LHGF}
Let $X\in\Simp$. The geometric realization $|X|$ of $X$ inherits the structure of a filtered space $|X|_*$, where $|X|_m$ is taken to be the geometric realization of the $m$-dimensional skeleton of $X$, for every $m$. Filtered homotopy equivalence of singular cubes in this case is thus called thin homotopy rel vertices. We consider the path $\omega$-groupoid $P^1\rho(|X|)$. $P^1\rho(|X|)$ captures homotopical properties of the path groupoid of $X$ with its skeletal filtration, e.g. 0-cubes are paths on $|X|_1$ (modulo thin homotopy), 1-cubes correspond to equivalence classes of homotopies among them constrained to $|X|_2$ By Lemma \ref{lem:PathsInternalGroupoid} $P^1\rho(|X|)$ is a groupoid internal to $\omega\Gpd$.
\begin{definition}\label{def:initiongaugefield}
Let $X\in\Simp$ be contractible. Let $K\in\Grp_{\omega\Gpd}$. A \textbf{local lattice higher gauge field (local LHGF)} on $X$ with gauge internal group $K$, is a morphism:

\[A:P^1\rho(|X|)\to BK\]
in $\Gpd_{\omega\Gpd}$.
\end{definition}
\noindent By Lemma \ref{lem:InternalGroups}, $\lo$s can be taken to take values on the internal group $B\rho(G)$ associated to a filtered topological group $G$. We will thus alternate between $\lo$s evaluated on general internal groups $K$ and on internal groups of the form $\rho(G)$ throughout the paper without much warning. Of chief interest to us will be the case in which the filtered group $G$ is a Lie group, and in most applications we will assume the trivial filtration on $G$. Given $X\in \Simp$ and $K\in \Grp_{\omega\Gpd}$, we will write $\fieldsXK$ for the set of $\lo$s on $X$, with gauge internal group $K$. We will write $\fieldsXG$ for the set of $\lo$s on $X$ with gauge group $\rho(G)$. \textit{Note:} Morphisms as in Definition \ref{def:initiongaugefield} can be defined without the assumption that the geometric realization $|X|$ of the simplicial set $X$ is contractible, but we will reserve the name $\lo$, and the notation $\fieldsXK$ for when this happens. In Section \ref{s:LocaltoGlobal} we drop this assumption and define global LHGF.
%New subsection
\subsection{$\lo$s on paths}\label{ss:LHGF on paths}
\noindent Vertices of $P^1\rho(|X|)$ are thin homotopy equivalences of singular paths on the 1-skeleton $X_1$ of $X$, having endpoints on $X_0$, i.e. vertices of $P^1\rho(|X|)$ are singular paths on 'continuous' space $|X|$, tracing a curve in $X_1$ and having endpoints vertices of $X$. We call such paths discrete paths. The usual concatenation operation $+_0$ of two such paths is encoded as their internal product operation $\odot$ in $P^1\rho(|X|)$. 
Given a filtered group $G$, a gauge field $A\in \fieldsXG$ associates to every discrete path $\gamma$ in $X$, an element $A(\gamma)\in G_0$. The unit element $1$ of $G_0$ is associated to the constant discrete path $ c_a$ on any vertex $a\in X_0$. Given two discrete paths $\gamma,\eta$ in $X$ such that the concatenation $\eta+_0\gamma$, i.e. the internal composition $\eta\odot\gamma$, is defined, the element $A(\eta+_0\gamma)$ of $G_0$ is equivalent to the internal composition $A(\eta)\odot A(\gamma)$, which in turn is equal to the product $A(\eta)A(\gamma)$ in $G_0$. Given a discrete path $\gamma$ in $X$, the element $A(-_0\gamma)$ of $G_0$ associated to the internal inverse $\gamma^{-1}$ of $\gamma$, which is equal to the inverse $-_0\gamma$, is equal to the inverse $A(\gamma)^{-1}$ in $G_0$. Finally, $\lo$s respect equivalence up to thin homotopy, if two paths $\gamma$ and $\eta$ are equivalent modulo thin homotopy, then the corresponding elements $A(\gamma)$ and $A(\eta)$ of $G_0$ are equal. We summarize this in the following table.
\begin{table}
	\caption{Action of $\lo$s on paths}
	\label{tab1}
	\begin{center}
		\begin{tabular}{ |p{6cm}|p{6cm}|  }
			\hline
			Data on $X$ & Data on $G$ \\
			\hline
			Discrete path $\gamma$ & Element $A(\gamma)\in G_0$\\
   Thin homotopy equivalence $\gamma\equiv\eta$ & Equality $A(\gamma)=A(\eta)$ in $G_0$\\
			Constant path on vertex $a$& Unit element $1$ of $G_0$\\
			Composite path $\eta+_0\gamma$ & Product $\gamma(\eta)\odot A(\gamma)$ in $G_0$ \\
			Path with reversed orientation $-_0\gamma$ & Inverse $A(\gamma)^{-1}$ in $G_0$\\

			\hline
		\end{tabular}
	\end{center}
\end{table}

\noindent The contents of table \ref{tab1} allow us to think of $\lo$s in $\fieldsXG$ as acting on discrete paths by coherently colouring paths with elements of the group $G_0$, in a way compatible with thin homotopy equivalence. Compare with \cite{Pfeiffer}.

\subsection{$\lo$s on squares}\label{ss:LHGFonsquares}
\noindent The data described in Table \ref{tab1} describes the action of a $\lo$ on the 0-dimensional data of $P^1\rho(|X|)$. In order to describe the action of a $\lo$ on 1-dimensional data we proceed as before. First observe that paths in $P^1\rho(|X|)$ are singular squares in $|X|$. Singular squares in $|X|$ are singular squares $\gamma$ on the continuous space $|X|$ such that the vertices of $\gamma$ are in $X_0$, the edges of $\gamma$ are discrete paths in $X_1$, and such that the image of $\gamma$ is contained in the 2-skeleton of $X$. We call such squares discrete squares. The 1-dimensional concatenation of paths in $P^1\rho(|X|)$ is the concatenation operation $+_1$ in direction 1 of $\rho(|X|)$. The concatenation operation $+_0$ of $\rho(|X|)$ is encoded as the internal composition $\odot$ of $P^1\rho(|X|)$.

A $\lo$ $A\in \fieldsXG$ associates to every discrete square $\gamma$ in $X$ a singular path $A(\gamma)$ in $\rho(G)$. Given a discrete square $\gamma$ in $X$, $A$ associates to the boundary paths in direction 1, $\partial_1^\alpha(\gamma)$, the source and target $\partial_0^\alpha A(\gamma)$ of the singular path $A(\gamma)$, and associates to the boundary paths in direction 0, $\partial_0^\alpha(\gamma)$, the internal source and target vertices $sA(\gamma)$ and $tA(\gamma)$, which are both equal to the unique vertex of $\textbf{0}$. Given a discrete path $\gamma$ in $X$, $A$ associates to the trivial discrete square $\epsilon_1(\gamma)$ the constant path on the element $A(\gamma)$ of $G_0$, and associates to the trivial discrete square $\epsilon_0(\gamma)$ the constant path on the unit element $1$ of $G_0$. Given two discrete squares $\gamma$ and $\eta$, if the concatenation $\eta+_1\gamma$ is defined, $A$ associates to $\eta+_1\gamma$ the concatenation of paths $A(\eta)+_0A(\gamma)$, and if the concatenation $\eta+_0\gamma$ is defined, $A$ associates to $\eta+_0\gamma$, the internal product $A(\eta)\odot A(\gamma)$, which is equal to the pointwise product $A(\eta)A(\gamma)$ in $G_1$ of the singular paths $A(\eta)$ and $A(\gamma)$. Given a discrete square $\gamma$, $A$ associates to the square $-_1\gamma$, with opposite orientation in direction $1$, the path $-_{0}A(\gamma)$, and associates to the square $-_0\gamma$ the internal inverse of $A(\gamma)$, i.e. the square $A(\gamma)^{-1}$ obtained by considering pointwise inverses in $G_1$. Finally, given two discrete squares $\gamma$ and $\eta$ in $X$, if $\gamma$ and $\eta$ are equivalent up to thin homotopies, then $A(\gamma)$ and $A(\eta)$ are thin homotopic equivalent singular paths in $G_1,G_0$. The data above is summarized in table \ref{tab1}.

\begin{table}\label{tb:dim1}
	\caption{Action of $\lo$s on squares}
	\label{tab2}
	\begin{center}
		\begin{tabular}{ |p{6cm}|p{6cm}|  }
			\hline
			Data on $X$ & Data on $G$ \\
			\hline
			Square $\gamma$ & Singular path $A(\gamma)$ in $(G_1,G_0)$\\
   Thin homotopy equivalence $\gamma\equiv\eta$ & $A(\gamma)\equiv A(\eta)$ in $\Pi_1(G_1,G_0)$\\
			1-faces $\partial_1^0(\gamma)$ and $\partial_1^1(\gamma)$ of $\gamma$ & Source and target of $A(\gamma)$ in $G_0$\\
			Degenerate square $\epsilon_1(\gamma)$ & Constant path on $A(\gamma)$\\
			Degenerate square $\epsilon_0(\gamma)$ & Constant path on $1$ of $G_0$\\
			Composite square $\eta+_1\gamma$ & Composition $A(\eta)+_0 A(\gamma)$ \\
			Composite square $\eta+_0\gamma$ & Pointwise product $A(\eta)\odot A(\gamma)$ in $G_1$\\
			Square with reversed orientation $-_1\gamma$ & Path $-_0A(\gamma)$ in $\Pi_1(G_1,G_0)$\\
			Square with reversed orientation $-_0\gamma$ & Pointwise inverse $A(\gamma)^{-1}$ in $G_1$\\

			\hline
		\end{tabular}
	\end{center}
\end{table}
\noindent As the contents of table \ref{tab1}, the contents of table \ref{tab2} allow us to think of $\lo$s in $\fieldsXG$ as acting on discrete squares by coherently coloring squares $\gamma$ with singular paths on $\rho(G)$, traced from the label of the bottom edge $\gamma$ to the label of the top edge of $\gamma$, and this is compatible with thin homotopy equivalence. Again, compare with \cite{Pfeiffer}. Tables describing the action of $\lo$s on cubes of any dimension are obtained analogously.

\subsection{Coarse graining and a complex of local LHGFs}\label{ss:ClosedLHGF}

\begin{obs}\label{obs:ClosedLHGF}
Let $X\in\Simp$ and $K\in\Grp_{\omega\Gpd}$. By proposition \ref{prop:EnrichedInternalGroupoids}, the set of $\lo$s on $X$ with gauge internal group $K$ naturally inherits the structure of an $\omega$-groupoid $[P^1\rho(|X|),BK]$, where    
\[[P^1\rho(|X|),BK]_m=\Homset_{\Gpd_{\omegagpd}}(P^{m+1}\rho(|X|),BK)\]

\noindent Considering the geometric realization $|[P^1\rho(|X|),BK]|$ of $[P^1\rho(|X|),BK]$ we obtain a convenient topological space (a CW-complex) whose homotopy type may be of interest. We come back to this theme at the end of Section \ref{s:LocaltoGlobal}. 
\end{obs} 

\begin{obs}\label{obs:coarsgraining}
Given $X,Y\in\Simp$. If $Y$ refines $X$, then $\rho(|X|)\subseteq \rho(|Y|)$. This inclusion defines a morphism $j:P^1\rho(|X|)\to P^1\rho(|Y|)$ in $\Gpd_{\omega\Gpd}$. Given $K\in\Grp_{\omegagpd}$, the pullback $j^\ast$ of $j$, defines a function $\fieldsYG\to \fieldsXG$, expressing the coarse graining of $\lo$s on $Y$ into $\lo$s on $X$.
\end{obs}

%subsectionGlobular
\subsection{Globular $\lo$s}\label{ss:globularLHGF}
\noindent We end this section with the following useful observation
\begin{obs}\label{obs:GlobularLHGF}
Let $X\in\Simp$. The $\infty$-groupoid $H\rho(|X|)$ associated to $\rho(|X|)$ is the homotopy $\infty$-groupoid $\rho^\bigcirc(|X|)$ of $|X|$, see \cite{BrownGlobular}. Higher dimensional paths in $\rho^\bigcirc(|X|)$ are thin homotopy equivalence classes of globular cubes. Given a group $K$ internal to $\omegagpd$, the $\infty$-groupoid $HK$ inherits from $K$ the structure of a group internal to $\inftygrpd$, and we can thus consider the shifted $\infty$-groupoid $H K[-1]$. Every $\lo$ $A\in \fieldsXK$ induces, by Proposition \ref{prop:equivalenceGlobularCubic}, a morphism
\[\hat{A}:\rho^\bigcirc(|X|)\to HK[-1]\]
\noindent We will call $\hat{A}$ the globular $\lo$ associated to $A$. 
\end{obs}

%Newsubsection
\section{Local LHGF and free groupoids}\label{ss:Pfeiffer}
In this section we compare $\lo$s in the sense of Definition \ref{def:initiongaugefield} and lattice and higher lattice gauge fields in the sense of \cite{Pfeiffer}. This after dropping smoothness conditions on gauge groups and gauge 2-groups respectively.
\subsection{Rank}\label{ss: Rank}
Let $K$ be a cubical set. A cube $a$ in $K_m$ is \textit{thin} if $a$ is of the form $\epsilon_i(b)$  or $\Gamma_i^\alpha(b)$ for some cube $b$ in $K_{m-1}$. If $K$ is an $\omega$-groupoid, a cube $a\in K_m$ is \textit{algebraically thin}, if $a$ can be subdivided into thin cubes. A cube $a$ of $K$ is of dimension $m$ if $a\in K_m$ and $a$ is not algebraically thin. The rank $rkK$ of a cubical $\omega$-groupoid $K$ is the supremum of dimensions of cubes in $K$. The $\omega$-groupoid $\rhogpdX$ associated to a simplicial set $X$ of dimension $m$ is clearly of rank $m$. If an $\omega$-groupoid $K$ is of rank $m\geq 0$, then the path $\omega$-groupoid $P^1K$ is of rank $max\left\{0,m-1\right\}$. Write $\omegagpd(m)$ for the full subcategory of $\omegagpd$ generated by $\omega$-groupoids of rank $\leq m$. For every $m\geq 0$, there exists a truncation functor $T_m:\omegagpd\to \omegagpd(m)$ associating to every $\omega$-groupoid $K$, the $\omega$-groupoid $K(m)$ of rank $\leq m$ obtained by removing, from $K$, all non-algebraically thin cubes of rank $\geq m$. There are obvious equivalences $\omegagpd(0)\cong\Set$, \label{2obsequiv}$\omegagpd(1)\cong\Gpd$, and \label{3obsequiv}$\omegagpd(2)\cong\dblGpd$. We can thus identify every $\omega$-groupoid $K$ of rank 0/1/2 with a set/groupoid/double groupoid with connections.

An analogous situation occurs for $\infty$-groupoids. $\inftygrpd(n)$ is defined in a way analogous to the way $\omegagpd(n)$ was defined, for every $\geq 0$ there is a truncation functor $T_n:\inftygrpd\to\inftygrpd(n)$, and there are equivalences $\inftygrpd(0)\cong\Set$, $\inftygrpd(1)\cong\Gpd$, and $\inftygrpd(2)\cong 2\Gpd$, where $2\Gpd$ denotes the category of strict 2-groupoids. We will identify every $\infty$-groupoid $K$ of rank 0/1/2 with a set/groupoid/2-groupoid.
%Nuew subsection
\subsection{Local LHGFs and free groupoids}\label{ss:LHGFGPfeiffer}
Let $X\in\Simp$ be a simplicial set of dimension 1. We think of $X$ as the 1-skeleton of a larger simplicial set $Y$ triangulating a manifold. $\catG^X$ denotes the free groupoid generated by $X$. The collection of objects of $\catG^X$ is the set $X_0$ of vertices of $X$, and the set of isomorphisms of $\catG^X$ is the set of formal sequences:
\[e_k^{\epsilon_k}\cdots e_1^{\epsilon_1}\]
where $e_k^{\epsilon_k}\cdots e_1^{\epsilon_1}$ are adjacent edges in $X$, and $\epsilon_i\in\left\{\pm 1\right\}$, and where any formal composition of the form $ee^{-1}$ or $e^{-1}e$ is identified with the empty word. Given a Lie group $G$, a lattice gauge field is usually defined as a functor $A:\catG^X\to BG$ where $BG$ is the delooping Lie groupoid of $G$, see \cite{Pfeiffer}. The following proposition says that this definition coincides with Definition \ref{def:initiongaugefield}.
\begin{prop}\label{prop:gaugefieldsdim1}
Let $X\in\Simp$. Let $G$ be a topological group. There is a bijection
\[\fieldsXG\cong \Homset_{\Gpd}(\catG^X,BG)\]
\end{prop}
\begin{proof}
The subcategory of $\Simp$ generated by simplicial sets of dimension 1 is equal to the subcategory of $\CubeSet$ generated by cubical sets of dimension 1. The simplicial set $X$ of dimension 1, can thus be considered as a simplicial set, and as a cubical set simultaneously. The groupoid $\catG^X$ is free on $X$, in the category $\Gpd$, while the $\omega$-groupoid $\rho(|X|)$ is free on the cubical set $X$, in the category $\omegagpd$. c.f. \cite[Proposition 14.6.2]{BrownBook}. The $\infty$-groupoid $\rho^\bigcirc(|X|)$ is thus free on $X$ in the category $\inftygrpd$. From the fact that $X$ is of dimension 1 it follows that all the cubes of $\rho(|X|)$ of dimension $\geq 2$ can be subdivided into degenerate cubes. It follows that all the cells of dimension $\geq 2$ of $\rho^\bigcirc(|X|)$ are trivial. The subcategory of $\inftygrpd$ generated by such $\infty$-groupoids is equivalent to $\Gpd$. The groupoid $\rho^\bigcirc(|X|)$ is thus free on $X$, in the category $\Gpd$. We conclude that, as groupoids $\catG^X$ and $\rho^\bigcirc(|X|)$ are isomorphic. 

From the fact that every cell of dimension $\geq 2$ of $\rho^\bigcirc(|X|)$ is trivial it follows that the image $A(\langle a\rangle)$ of any homotopy cell $\langle a\rangle$ in $\rho^\bigcirc(|X|)$ under the induced globular local LHGF $\hat{A}:\rho^\bigcirc(|X|)\to\rho^\bigcirc(G)$ associated to a local LHGF $A:P^1\rho(|X|)\to\rho(G)$, is trivial. From this and from the fact that the subgroupoid of $\rho^\bigcirc(G)[-1]$ generated by squares of dimensions 0 and 1, is equal to the delooping groupoid $BG$, it follows that there exists a bijection from $\Homset_{\inftygrpd}(\rho^\bigcirc(|X|),\rho^\bigcirc(G)[-1])$ to $\Homset_\Gpd(\rho^\bigcirc(|X|),BG)$. The proposition follows from this and from the comments in \ref{ss:globularLHGF}.
\end{proof}
\subsection{Local LHGFs and free 2-groupoids}\label{ss:Pfeifferdim2}
Let $X\in\Simp$ be a simplicial set of dimension 2. We think of $X$ now as the 2-skeleton of a larger simplicial set $Y$ triangulating a manifold. In that case every homotopy cube $\langle a\rangle\in\rhogpdX$ of dimension $\geq 3$ is algebraically thin. Thus, the image of a local LHGF $A\in \fieldsXK$ is contained in the 2-truncation $T_2K$ of $K$. We have a natural bijection $\fieldsXK\cong \mathcal{M}(X,T_2K)$ and we thus consider LHGFs on $X$ to be evaluated on a special double groupoid with connections $K$, in which case $H K$ is a 2-group. In \cite{Pfeiffer} higher gauge fields are defined as strict 2-functors from the free 2-groupoid $\catG^X$ generated by $X$, to a given Lie 2-group $K$. We prove that if we drop the smoothness requirement on the target 2-group $K$, higher lattice gauge field as defined above coincide with local LHGFs on simplicial sets of dimension 2.

Write $\Sigma:\Delta\to\CubeSet$ for the functor from the simplicial category $\Delta$ to $\CubeSet$ regarding $[n]$ as a cubical set with standard degeneracies for every $n$. $\Sigma$ thus acts as the identity on $[1]$ and acts on $[2]$ by associating it the singular cubical set on $\Delta$ with a standard degeneracy on the edge $\partial_0^0[0,1]^2$. Let $\Sigma^*:\Simp\to\CubeSet$ be the cocontinuous extension of $\Sigma$ along the Yoneda embedding. With this notation $\Sigma^*X$ is the standard cubulation of the simplicial set $X$ of dimension 2. We prove the following theorem.
\begin{thm}\label{thm:Pfeifferdim2}
    Let $X\in\Simp$ be a simplicial set of dimension 2. Let $K\in 2\Gpd$ be a 2-groupoid. There is a natural bijection:
    \[\Homset_{2\Gpd}(\catG^X,K)\cong \Homset_{2\Gpd}(\rhogpdglobsigmaX,K)\]
\end{thm}
\begin{proof}
In order for the natural bijection in the statement to make sense we first observe that from the assumption that $X$ has dimension 2, it follows that $\rhogpdglobX$ is of rank 2, and thus can be considered as a 2-groupoid. The cubulation $\Sigma^*X$ of $X$ is such that $|X|=|\Sigma^*X|$. By \cite[Theorem 5.3]{BrownMosa} there is a natural bijection 
\[Hom_{2\Gpd}(\rhogpdglobsigmaX,K)\cong\Homset_{\dblGpd}(\rhogpdsigmaX,QK)\]
where $QK\in\dblGpd$ is the Ehresmann quintets special double groupoid associated to $K$ \cite{DawsonPare} and where $\rhogpdsigmaX$ is considered as a special double groupoid with connections (which can be done since $X$ is of dimension 2). The comments made in subsection \ref{ss: Rank} imply the existence of a natural bijection
\[Hom_{\dblGpd}(\rhogpdsigmaX,QK)\cong\Homset_{\omegagpd}(\rhogpdsigmaX,T_2QK)\]
where now $\rhogpdsigmaX$ is considered as an $\omega$-groupoid. By \cite[Proposition 14.6.2]{BrownBook} there now exists a natural bijection
\[Hom_{\omegagpd}(\rhogpdsigmaX,T_2K)\cong\Homset_{\CubeSet}(\Sigma^*X,UT_2QK)\]
where $UT_2QK$ represents the cubical set underlying $T_2QK$. There is an obvious natural bijection
\[Hom_{\CubeSet}(\Sigma^*X,UT_2QK)\cong\Homset_{\Simp}(X,UK)\]
where $UK$ is the 2-dimensional simplicial set underlying $K$. Finally, the fact that $\catG^X$ is free on $X$ provides a natural bijection
\[Hom_{\Simp}(X,UK)\cong\Homset_{2\Gpd}(\catG^X,K)\]
which proves the claim of the theorem.
\end{proof}
\begin{cor}\label{cor:Pfeifferdim2}
 Let $X\in\Simp$ be a simplicial set of dimension 2. Let $K\in \Grp_{\omegagpd}$ be group internal to $\omegagpd$. There exists a narutal bijection
 \[\fieldsXK\cong\Homset_{2\Gpd}(\catG^X,H T_2K[-1])\]
\end{cor}

\section{Gauge orbits}\label{s:GaugeEquivalence}
In this section we study local gauge transformations and gauge equivalence between $\lo$s. The notion of local gauge transformation presented in this section is modelled after gauge transformation as appearing in \cite{MenesesZapata1,MenesesZapata2}
\subsection{Local gauge transformations}\label{ss:Gaugetransf}
\begin{definition}\label{def:sigmamodel}
Let $X\in\Simp$, let $G\in\FGp$. A higher homotopy $G$-valued field on $X$, is a morphism $u:\rho(|X|)\to \rho(G)$ in $\omegagpd$. We will write $\model$ for the set of higher homotopy $G$-valued fields on $X$.
\end{definition}
%Definition
\begin{definition}\label{def:gaugetransf}
Let $X\in\mbox{\textbf{SimpSet}}$. Let $G\in\mbox{\textbf{FGrp}}$. Let $A,B$ be $\lo$s on $X$, with gauge filtered group $G$. A local gauge transformation from $A$ to $B$ is a $u\in\model$, such that for every $n\geq 1$, and every homotopy $n-1$-cube $\langle a\rangle\in P^1\rho(|X|)^1_{n-1}$, i.e for every homotopy $n$-cube $\langle a\rangle\in\rho(|X|)_{n}$, the following equation holds:
\begin{equation}\label{eq:gaugetransfgeneral}
    u_n(s\langle a\rangle)\odot A_{n-1}^1(\langle a\rangle)\odot u_n(t\langle a\rangle)^{-1}=B_{n-1}^1(\langle a\rangle).
\end{equation}
In the above case we will write $u:A\rightsquigarrow B$. The $\lo$s $A$ and $B$ are gauge equivalent if there exists a local gauge transformation $u:A\rightsquigarrow B$.
\end{definition}
\begin{obs}\label{obs:naturaliso}
    Let $X\in\mbox{\textbf{SimpSet}}$. Let $G\in\mbox{\textbf{FGrp}}$. Let $A,B$ be $\lo$s on $X$, with gauge filtered group $G$. A local gauge transformation from $A$ to $B$ is a natural isomorphism $u:A\Rightarrow B$. In \cite{Pfefiffer, BaezSchreiber} pseudonatural transformations between gauge fields are considered. We only consider strict natural transformations in this paper. Weaker notions will be studied in subsequent work.
\end{obs}
%FirstExample
\begin{ex}[Gauge transformations on vertices]\label{ex:gaugetransfdim0}
Let $X\in\mbox{\textbf{SimpSet}}$. Let $G\in\mbox{\textbf{FGrp}}$. Let $A,B$ be $\lo$s on $X$, with gauge filtered group $G$, and let $u:A\rightsquigarrow B$ be a local gauge transformation. If $x\in X_0$ is a vertex of $X$, i.e., if $x\in\rho(|X|)_0$, then $u_0(x)$ is an element of $G_0$. Given a simplicial path $\gamma:x\to y$ in $X$, equation (\ref{eq:gaugetransfgeneral}) translates to the equation:
\begin{equation}\label{eq:gaugetransdim0}
  u_0(x)\odot A_1^1(\langle \gamma\rangle)\odot u_0(y)^{-1}=B_{1}^1(\langle \gamma\rangle).  
\end{equation}
which translates to the following equation in the group $G_0$:
\begin{equation}\label{eq:gaugetransdim0inG}
  u_0(x)A_1^1(\langle \gamma\rangle) u_0(y)^{-1}=B_{1}^1(\langle \gamma\rangle).   
\end{equation}
The value of $B$ on a discrete path $\gamma$ can thus be recovered from the value of $A$ on $\gamma$ by conjugating by the values of $u$ on the source and target vertices $x$ and $y$, respectively. 
\end{ex}
%SecondExample
\begin{ex}[Gauge transformations on paths]\label{ex:gaugetransfdim1}
Let $X\in\mbox{\textbf{SimpSet}}$. Let $G\in\mbox{\textbf{FGrp}}$. Let $A,B$ be $\lo$s on $X$, with gauge filtered group $G$, and let $u:A\rightsquigarrow B$ be a local gauge transformation. If $\langle\gamma\rangle\in \rho(|X|)_1$ is a homotopy simplicial path on $X$ from a vertex $x$ to a vertex $y$, then $u_1(\langle\gamma\rangle)$ is a homotopy path:
\[u_1(\langle\gamma\rangle):u_0(x)\to u_0(y)\]
in $\rho(G)_1$. Given a simplicial square $S$ in $X$, equation (\ref{eq:gaugetransfgeneral}) translates to the equation:
\begin{equation}\label{eq:gaugetransdim1}
  u_1(s\langle S\rangle)\odot A_2^1(\langle S\rangle)\odot u_0(t\langle S\rangle)^{-1}=B_{1}^1(\langle S\rangle).  
\end{equation}
The value of $B$ on $\langle S\rangle$ can thus be recovered from the value of $A$ on $\langle S\rangle$ by conjugating by the value of $u_1$ at the left and right edge homotopy paths $s\langle S\rangle$ and $t\langle S\rangle$ above. Moreover, the condition of $u$ being a morphism of $\omega$-groupoids says that $u$, and thus Equation (\ref{eq:gaugetransdim1}) is compatible with composition of homotopy paths in $\rho(|X|)_1$.
\end{ex}
\begin{obs}\label{obs:pathofpathstransf}
Let $X\in\mbox{\textbf{SimpSet}}$. Let $G\in\mbox{\textbf{FGrp}}$. Let $A,B$ be $\lo$s on $X$, with gauge filtered group $G$, and let $u:A\rightsquigarrow B$ be a local gauge transformation. We can think of a homotopy square $\langle S\rangle$ in $\rho(|X|)$ as a homotopy path of homotopy paths $\langle \gamma_z\rangle$ parametrized by points in $[0,1]$. The gauge transformation $u$ is thus a homotopy path of gauge transformations on homotopy paths localized at different parameter points of $s\langle S\rangle$. Schematically, we picture $\langle S\rangle$ as:
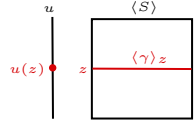
\begin{center}

\tikzset{every picture/.style={line width=0.75pt}} %set default line width to 0.75pt        

\begin{tikzpicture}[x=0.75pt,y=0.75pt,yscale=-1,xscale=1]

%Shape: Square [id:dp5603623942069555] 
\draw   (279.6,100.6) -- (329.6,100.6) -- (329.6,150.6) -- (279.6,150.6) -- cycle ;
%Straight Lines [id:da21415689325372056] 
\draw [color={rgb, 255:red, 214; green, 2; blue, 6 }  ,draw opacity=1 ]   (279.9,125.5) -- (329.3,125.7) ;
%Straight Lines [id:da9872555448937608] 
\draw [color={rgb, 255:red, 0; green, 0; blue, 0 }  ,draw opacity=1 ]   (259.9,99.5) -- (260.2,150.7) ;
%Straight Lines [id:da2355644228064251] 
\draw [color={rgb, 255:red, 214; green, 2; blue, 6 }  ,draw opacity=1 ]   (260.05,125.1) ;
\draw [shift={(260.05,125.1)}, rotate = 0] [color={rgb, 255:red, 214; green, 2; blue, 6 }  ,draw opacity=1 ][fill={rgb, 255:red, 214; green, 2; blue, 6 }  ,fill opacity=1 ][line width=0.75]      (0, 0) circle [x radius= 1.34, y radius= 1.34]   ;

% Text Node
\draw (271.28,123.56) node [anchor=north west][inner sep=0.75pt]  [font=\tiny,color={rgb, 255:red, 214; green, 2; blue, 6 }  ,opacity=1 ]  {$z$};
% Text Node
\draw (297,90) node [anchor=north west][inner sep=0.75pt]  [font=\tiny,color={rgb, 255:red, 0; green, 0; blue, 0 }  ,opacity=1 ]  {$\langle S\rangle $};
% Text Node
\draw (297.46,115) node [anchor=north west][inner sep=0.75pt]  [font=\tiny,color={rgb, 255:red, 214; green, 2; blue, 6 }  ,opacity=1 ]  {$\langle \gamma \rangle _{z}$};
% Text Node
\draw (237,121) node [anchor=north west][inner sep=0.75pt]  [font=\tiny,color={rgb, 255:red, 214; green, 2; blue, 6 }  ,opacity=1 ]  {$u(z)$};
% Text Node
\draw (254,92.76) node [anchor=north west][inner sep=0.75pt]  [font=\tiny,color={rgb, 255:red, 0; green, 0; blue, 0 }  ,opacity=1 ]  {$u$};

\end{tikzpicture}
\end{center}
and $u$ can thus be thought as a path of gauge transformations of the paths $\gamma_z$.
\end{obs}
\noindent Definitions of gauge transformation between higher gauge fields satisfying conditions of a more categorical/differential nature appear in \cite{Pfeiffer,BaezSchreiber}. \textbf{Question:} How does the notion of local gauge transformation as appearing in \ref{def:gaugetransf} relates to that of \cite{Pfeiffer}? We have not attempted a version of local gauge transformation that might admit a higher categorical interpretation. The obvious choice of this would be a natural transformation between morphisms of internal groupoids. The study of such transformations and their relation to local gauge transformations as presented here will be studied in future work.

\subsection{Groupoid structure}\label{subsec:groupoid}
\begin{obs}\label{obs:gaugetransfcomposition}
Let $X\in\mbox{\textbf{SimpSet}}$. Let $G\in\mbox{\textbf{FGrp}}$. Let $A,B,C$ be $\lo$s on $X$, with gauge filtered group $G$, and let $u:A\rightsquigarrow B,v:B\rightsquigarrow C$ be local gauge transformations. If we write $u\odot v$ for the morphism from $\rho(|X|)$ to $\rho(G)$ associating, to every $\langle a\rangle\in \rho(|X|)_n$, the product $u_n(\langle a\rangle)\odot v_n(\langle a\rangle)$, is a gauge transformation $A\rightsquigarrow C$: The commposite map $u\odot v$ is obviously a morphism of $\omega$-groupoids. Now, given $\langle a\rangle\in \rho(|X|)_n$, applying equation (\ref{eq:gaugetransfgeneral}) to $v$ we obtain:
\begin{equation}\label{eq:gaugetransfcomposition1}
    v_n(s\langle a\rangle)\odot B_{n}^1(\langle a\rangle)\odot v_n(t\langle a\rangle)^{-1}=C_n^1(\langle a\rangle).
\end{equation}
Again, from equation (\ref{eq:gaugetransfgeneral}), now applied to $u$, we obtain the equation:
\begin{equation}\label{eq:gaugetransfcomposition2}
     u_n(s\langle a\rangle)\odot A_{n}^1(\langle a\rangle)\odot u_n(t\langle a\rangle)^{-1}=B_n^1(\langle a\rangle).
\end{equation}
Substituting equation (\ref{eq:gaugetransfcomposition2}) in equation (\ref{eq:gaugetransfcomposition1}) we obtain the following equation:
\begin{equation}\label{eq:gaugetransfcomposition3}
    u_n(s\langle a\rangle)\odot v_n(s\langle a\rangle)\odot A_{n}^1(\langle a\rangle)\odot v_n(t\langle a\rangle)^{-1}\odot u_n(t\langle a\rangle)^{-1}=C_n^1(\langle a\rangle),
\end{equation}
which translates to the equation:
\begin{equation}\label{eq:gaugetransfcomposition3}
    u_n\odot v_n(s\langle a\rangle)\odot A_{n}^1(\langle a\rangle)\odot (u_n\odot v_n)^{-1}(t\langle a\rangle)^{-1}=C_n^1(\langle a\rangle).
\end{equation}
The composite map $u\odot v$ is thus a lattice gauge transformation.
\end{obs}
\noindent The following observation follows from an argument analogous to Observation \ref{obs:gaugetransfcomposition}.
\begin{obs}\label{obs:gaugetransfinverses}
Let $X\in\mbox{\textbf{SimpSet}}$. Let $G\in\mbox{\textbf{FGrp}}$. Let $A$ be a $\lo$ on $X$, with gauge group $G$. The constant function $\textbf{1}:\rho(|X|)\to\rho(G)$ associating to every homotopy cube $\langle a\rangle\in \rho(|X|)_n$ the constant homotopy cube $\textbf{1}$ on the identity element $1$ of $G_0$. Thus defined $\textbf{1}$ is a local gauge transformation $\textbf{1}:A\rightsquigarrow 1$. Moreover, given a second $\lo$ $B$ on $X$, with gauge group $G$, and a gauge transformation $u:A\rightsquigarrow B$, the map $u^{-1}$ associating to every homotopy cube $\langle a\rangle\in\rho(|X|)$, the inverse homotopy cube $u_n(\langle a\rangle)^{-1}$ with respect to the internal product $\odot$ on $\rho(G)$, is a gauge transformation $u^{-1}:B\rightsquigarrow A$ satisfying the equation $u\odot u^{-1}=\textbf{1}=u^{-1}\odot u$.
\end{obs}
\noindent From Observations \ref{obs:gaugetransfcomposition} and \ref{obs:gaugetransfinverses} it follows that the collection of LHGFs on a simplicial set $X$, with give gauge group $G$ forms a groupoid.
%Notation Groupoid
\begin{definition}\label{def:GroupoidGauge}
Let $X\in\mbox{\textbf{SimpSet}}$. Let $G\in\mbox{\textbf{FGrp}}$. $\textbf{Gauge}_X^G$ is the groupoid with $\lo$s on $X$ with gauge group $G$, with gauge transformations as morphisms, and with composition operation as in Observation \ref{obs:gaugetransfcomposition}. Two $\lo$s $A$ and $B$ are gauge equivalent if $A$ and $B$ belong to the same connected component, i.e. if they are isomorphic, in $\textbf{Gauge}_X^G$.
\end{definition}
\begin{obs}\label{obs:actiongaugetransf}
Let $X\in\Simp$. Let $G\in\FGp$. $\model$ inherits, from the internal group structure of $\rho(G)$, the structure of a group. We refer to $\model$, with this structure as the group of local $G$-valued gauge transformations on $X$. The group $\model$ acts on the set $\fieldsXG$ of $\lo$s by conjugation. The groupoid $\textbf{Gauge}_X^G$ is precisely the action groupoid $\fieldsXG\rtimes \model$ of this action and two $\lo$s $A$ and $B$ are gauge equivalent if and only if they are in the same orbit under the above action.
\end{obs}
\subsection{Local gauge transformations on globular LHGFs}
\begin{obs}\label{obs:GlobularVertices}
Let $X\in\mbox{\textbf{SimpSet}}$. Let $G\in\mbox{\textbf{FGrp}}$. Let $A,B$ be LHGFs on $X$, with gauge filtered group $G$, and let $u: A\rightsquigarrow B$ be a local gauge transformation. Consider the globular LHGFs $\hat{A}$ and $\hat{B}$. These are morphisms $\hat{A},\hat{B}:\rho^\bigcirc(|X|)\to \rho(G)[-1]$, in $\infty$\textbf{Gpd}. Given a globular homotopy cube $\langle a\rangle\in\rho^\bigcirc(|X|)_n$. Since $\langle a\rangle$ is globular, $s\langle a\rangle,t\langle a\rangle\in Im \epsilon_0^n$, then $u( s\langle a\rangle)$ and $u(t\langle a\rangle)$ depend only on the value of $u$ on the vertices $\epsilon_0^n \partial_0^{0^n}\langle a \rangle$ and $\epsilon_0^n \partial_0^{1^n}\langle a \rangle$. The action of $u$ on $\hat{A}$ and $\hat{B}$ is thus completely determined by the action of $u$ evaluated on vertices of $X$.
\end{obs}
\begin{definition}\label{def:globulargaugetransf}
Let $X\in\mbox{\textbf{SimpSet}}$. Let $G\in\mbox{\textbf{FGrp}}$. Let $A,B$ be globular LHGFs on $X$, with gauge filtered group $G$. A local gauge transformation $u:A\rightsquigarrow B$ is a function $u:X_0\to G_0$ such that, for every homotopy 2-cell $\langle a\rangle\in \rho^\bigcirc(|X|)_n$, the following equation holds in $\rho^\bigcirc(G)[-1]$:
\begin{equation}\label{eq:globulargaugetransf}
    u(s^n\langle a\rangle)+_0 A(\langle a\rangle)-_0u(t^n\langle a\rangle)=B(\langle a\rangle).
\end{equation}
\end{definition}
\begin{ex}\label{ex:globulargaugetransfdim1}[Gauge transformations on paths]
Let $X\in\mbox{\textbf{SimpSet}}$. Let $G\in\mbox{\textbf{FGrp}}$. Let $A,B$ be globular LHGFs on $X$, with filtered gauge group $G$, and $u:A\rightsquigarrow B$ be a gauge transformation. Observe that every homotopy path in $\rho(|X|)_1$ is globular. Let $\langle\gamma\rangle\in \rho^\bigcirc (|X|)_1$ be a path from $x$ to $y$. In that case $A(\langle \gamma\rangle),B(\langle\gamma\rangle)$, $u(x)$ and $u(y)$ are all elements of $G_0$, and equation (\ref{eq:globulargaugetransf}) becomes the following equation:
\begin{equation}\label{eq:globulargaugetransdim0}
  u(x)+_0 A_1(\langle \gamma\rangle)-_0 u(y)=B_{1}(\langle \gamma\rangle),  
\end{equation}
in $\rho^\bigcirc(|X|)$, which translates to the following equation in the group $G_0$:
\begin{equation}\label{eq:globulargaugetransdim0inG}
  u(x)A_1(\langle \gamma\rangle) u(y)^{-1}=B_{1}(\langle \gamma\rangle).   
\end{equation}
Contrast this with equation (\ref{eq:gaugetransdim0}).
\end{ex}
\begin{ex}\label{ex:globulargaugetransfonglobes}[Gauge transformations on 2-cells]
Let $X\in\mbox{\textbf{SimpSet}}$. Let $G\in\mbox{\textbf{FGrp}}$. Let $A,B$ be globular LHGFs on $X$, with gauge group $G$, and $u:A\rightsquigarrow B$ be a gauge transformation. Let $\langle S\rangle\in\rho^\bigcirc(|X|)_2$ be a homotopy 2-cell such that $s\langle S\rangle=\gamma,t\langle S\rangle=\eta, s^2\langle S\rangle=x$ and $t^2 \langle S\rangle=y$. In that case $A(\langle S\rangle)$ and $B(\langle S\rangle)$ are 2-cells in $\rho^\bigcirc(G)[-1]$ with source and target $A(\langle\gamma\rangle)$ and $A(\langle\eta\rangle)$, and $B(\langle\gamma\rangle)$ and $B(\langle\eta\rangle)$ respectively. Now, $u(x),u(y)$ are elements of $G_0$ and thus $0$-cells in $\rho^\bigcirc(G)$. The identities $i(u(x))$ and $i(u(y))$ are paths in $\rho^\bigcirc(G)$, and thus are 2-cells in the shifted $\infty$-groupoid $\rho^\bigcirc(G)[-1]$. Equation (\ref{eq:globulargaugetransf}) becomes the following equation:
\begin{equation}\label{eq:globulargaugetransfonglobes}
    u(x)+_0 A(\langle S\rangle)-_0u(y)=B(\langle S\rangle).
\end{equation}
Given a 2-cell $\langle S\rangle$ in $\rho^\bigcirc(|X|)$, as below, with horizontal source and target $s^2(\langle S\rangle)$ and $t^2(\langle S\rangle$, the value of $B$ at $\langle S\rangle$ can thus be recovered from the value of $A$ at $\langle S\rangle$ by conjugating, in $\rho^\bigcirc(G)[-1]$ by the value of $u$ at the vertices $x$ and $y$. Contrast this with Example \ref{ex:gaugetransfdim1} where to specify the action of $u$ on a square as above it was necessary to specify the value of $u$ on homotopy paths and not just on vertices.
\end{ex}
\noindent Observations analogous to \ref{obs:actiongaugetransf} hold for local gauge transformations of globular $\lo$s.

\section{The local to global problem for LHGF}\label{s:LocaltoGlobal}
\noindent In this section we make use of ideas developed in Section \ref{s:GaugeEquivalence} to provide a definition of global LHGF. We provide a local to global principle for globular global LHGFs satisfying certain compatibility conditions. \textit{Note:} In this section we drop the requirement that the geometric realization $|X|$ of a simplicial set $X$ must be contractible. We instead make a choice of good cover for $X$ and focus on LHGFs defined with respect to the given cover.
\subsection{Global LHGFs}\label{ss:gobalLHF}

\noindent From Section \ref{s:LHGF} on it was assumed that the geometric realization $|X|$ of the base simplicial set $X$ of a LHGF is contractible. In this section we consider instead simplicial sets $X$ admitting a good cover $\{X_1,...,X_n\}$, i.e. a cover where geometric realization of any intersection of the $X_i$'s is contractible; such that $X$ decomposes as the gluing of the simplexes $X_1,...,X_n$. We express this by requiring the an equation of the form $X=colimX_i$ holds. A typical example is the cover by simplexes of maximal dimension of $X$. The reader is welcome to think of this particular case throughout this section. We start by defining global cubical LHGFs with respect to a good cover of a not-necessarily-contractible simplicial set $X$. 
%Say something about examples in the intro
\begin{comment}
\end{comment}
A part from the data coming from the decomposition $X=colim X_i$ of $X$ together with the condition that $\left\{X_1,...,X_n\right\}$ forms a good cover, our definition of global LHGF assumes a local trivialization over each $X_i$, together with knowledge of the corresponding transition function $\psi_{ij}:\rho(|X_{ij}|)\to \rho(G)$ at each intersection $X_{ij}= X_i \cap X_j$.

\begin{definition}\label{def:globalLHGF}
Let $X\in\Simp$. Let $\{ X_1,...,X_n\}$ be a good cover of $X$.
Let $G$ be a Lie group. 
Let $\{ \psi_{ij}\in\sigma(X_{ij},G) \}$ be a system of transition functions 
determining a Cech cocycle. 
A global LHGF $A$ on the pair $(\{X_i\}, \{ \psi_{ij} \})$, with gauge group $G$, is an $n$-tuple $(A_i)$ of local LHGFs $A_i\in\mathcal{M}(X_i,G)$ such that the following compatibility condition holds for any homotopy cube $\langle a\rangle \in \rho (X_{ij})$ 

	\begin{equation}\label{eq:compatibilityglobal}
	A_i (\langle a\rangle) =  \psi_{ij}(s \langle a\rangle) \odot A_j (\langle a\rangle) \odot (\psi_{ij}(t \langle a\rangle))^{-1}. 
	\end{equation}
 \noindent We will write $\mathcal{M}_{gl}(\{ X_i\}, \{ \psi_{ij} \},G)$ for the set of global LHGFs on $(\{ X_i\}, \{ \psi_{ij} \})$ with gauge group $G$.
\end{definition}

%Definition
\begin{definition}\label{def:gaugetransfONTrFns}
Let $X\in\mbox{\textbf{SimpSet}}$. Let $G\in\mbox{\textbf{FGrp}}$. Let $\{ \psi_{ij} \}$ be a system of transition functions 
determining a Cech cocycle. 
A local gauge transformation $u_k$ over $X_k$ 
transforms the system of transition functions to the new system 
$\{ \psi'_{ij} \}$ determined as follows: 
if $i,j \neq k$ then $\psi'_{ij}= \psi_{ij}$; 
if $i=k$ then $\psi'_{ij}= u_i\odot \psi_{ij}$; 
if $j=k$ then $\psi'_{ij}= \psi_{ij}\odot u_j^{-1}$. 
\end{definition}

Notice that the compatibility condition in Definition \ref{def:globalLHGF} is invariant under gauge transformations. 

\begin{obs}\label{obs:GlobularGlobal}
Globular global LHGFs with respect to a pair $(\{X_i\},\{\psi_{ij}\})$ as in Definition \ref{def:globalLHGF} are defined analogously. Notice however that in that case, the compatibility conditions \ref{eq:compatibilityglobal} depend only on the restriction of the transition functions to the vertices of $X$, which do not characterize a Cech cocycle. In the next section we will show, by proving that globular global LHGFs induce Extended Lattice Gauge Fields \cite{MenesesZapata1,MenesesZapata2}, that the bundle structure can be extracted from a globular global LHGF. Any two choices of transition functions on $\{X_1,...,X_n\}$ are equivalent. We thus modify the notation appearing in Definition \ref{def:globalLHGF}. We will denote by $\mathcal{M}_{gl}^\bigcirc(X,G)$ the set of globular global LHGFs on the pair $(\{ X_i\}, \{ \psi_{ij} \})$ with gauge group $G$. 
\end{obs}

\subsection{Local to global}\label{ss:LocaltoGlobal}

\noindent Consider a good cover $\{X_1,...,X_n\}$ of $X\in\Simp$, together with transition functions evaluated at vertices $\{ \psi_{ij}|_{X_0} \}$.
Let $(A_i)$ be a globular global LHGF, with respect to the pair $(\{ X_i\},\psi_{ij})$ as above.
In this subsection we prove that, from the data of $(A_i)$, we can reconstruct a morphism $A\in \mbox{Hom}_{\inftygrpd}(\rhogpdglobX,\rho^\bigcirc(G)[-1])$. This solves the local to global problem for gauge fields. We begin with two observations.

\begin{obs}\label{1onVertices}
    Given two systems of transition functions evaluated at vertices, $\{ \psi_{ij}|_{X_0}\}, \{ \psi'_{ij}|_{X_0}\}$, it is clear from Definition \ref{def:gaugetransfONTrFns} that there is a collection of local gauge transformations $\{ u_i \}$ taking $\{ \psi_{ij}|_{X_0}\}$ to $\{ \psi'_{ij}|_{X_0}\}$. 
In particular, given a system of transition functions evaluated at vertices $\{ \psi_{ij}|_{X_0}\}$ there is a collection of local gauge transformations $\{ u_i \}$ taking it to the the constant $1$ function. We will write this as $\{ \psi_{ij}|_{X_0} = {\mathbf{1}}\}$. 

In the rest of this subsection we will restrict to the case in which we are given a system of transition functions such that $\{ \psi_{ij}|_{X_0} = {\mathbf{1}}\}$ because if we are given a system of transition functions not satisfying the condition, we can apply local gauge transformations to bring it to the simpler form. 
\end{obs}

\begin{obs}\label{obs:Global}
This Observation follows directly from Observation \ref{obs:GlobularGlobal}. 
Let $G$ be a Lie group. Let $X\in\Simp$ and let $\{X_1,X_2\}$ be a good cover of $X$. Let $\psi_{12}$ be a transition function, such that 
$\psi_{12}|_{Y_0}= {\mathbf{1}}$. In that case the compatibility condition of Definition \ref{def:globalLHGF} is expressed by the fact that the following is a pullback diagram

\[\mathcal{M}_{gl}^\bigcirc(X,G)\to\mathcal{M}^\bigcirc(X_1,G)\times \mathcal{M}^\bigcirc(X_2,G)\rightrightarrows\mathcal{M}^\bigcirc(Y,G)\]

\noindent More generally, if $\left\{X_1,...,X_n\right\}$ is a good cover of $X$, together with transition functions $\psi_{ij}$ such that, evaluated at vertices are equal the constant 1 function, then the set $\mathcal{M}_{gl}^\bigcirc(X, G)$ satisfies the equation 
 \begin{equation}
 \mathcal{M}_{gl}^\bigcirc(X, G)=lim\mathcal{M}^\bigcirc(X_i,G)
 \end{equation} 
\end{obs}

\noindent Given this result, it is reasonable to consider that morphisms in 
$\mathcal{M}_{gl}^\bigcirc(X,G)$ are 
globular global LHGFs. 
In the following theorem observe that the elements of the set of morphisms $\mbox{Hom}_{\inftygrpd}(\rhogpdglobX,\rho^\bigcirc(G)[-1])$ are not considered $\lo$s since we are not assuming the geometric realization $|X|$ is contractible.

\begin{thm}\label{thm:localglobal}
  Let $G$ be a Lie group. Let $X\in \Simp$. Let $\left\{X_1,...,X_n\right\}$ be a good cover of $X$, with transition functions $\psi_{ij}$, such that, evaluated at vertices are equal the constant 1 function, then the set $\mathcal{M}_{gl}^\bigcirc(X,G)$ satisfies the equation 
 \begin{equation}
 \mathcal{M}_{gl}^\bigcirc(X,G)\cong\mbox{Hom}_{\inftygrpd}(\rhogpdglobX,\rho^\bigcirc(G)[-1])
 \end{equation}
\end{thm}

\begin{proof}
    We prove that Hom$_{\inftygrpd}(\rhogpdglobX,\rho^\bigcirc(G)[-1])$ is the limit $lim\mathcal{M}^\bigcirc(X_i,G)$ making use of the HHSvK Theorem of \cite{BrownBook}. The theorem will follow from this and from Observation \ref{obs:Global}. To prove the claim above, we prove the following claim: Given a good cover $\{X_1,X_2\}$ of $X$, the diagram:
\begin{equation}
    Hom_{\inftygrpd}(\rhogpdglobX,\rho^\bigcirc(G)[-1])\to\mathcal{M}^\bigcirc(X_1,G)\times \mathcal{M}^\bigcirc(X_2,G)\rightrightarrows\mathcal{M}^\bigcirc(X_{12},G)
\end{equation}
    \noindent is a pullback diagram. The claim will follow inductively from this. If $\{X_1,X_2\}$ is a good cover of $X$, then the diagram 
   $$\begin{tikzpicture}[node distance=2cm, auto]
  \node (X1X2) {$X_{12}$};
  \node (X1) [below of=X1X2] {$X_1$};
  \node (X) [right of=X1] {$X$};
  \node (X2) [right of=X1X2] {$X_2$};
  \draw[->] (X1X2) to node {} (X2);
  \draw[->] (X1X2) to node [swap] {} (X1);
  \draw[->] (X1) to node [swap] {} (X);
  \draw[->] (X2) to node [swap] {} (X);
\end{tikzpicture}$$
    \noindent is a pushout diagram in $\Simp$. After making a choice of convenient category of spaces, the diagram:
   $$\begin{tikzpicture}[node distance=2cm, auto]
  \node (X1X2) {$|X_{12}|$};
  \node (X1) [below of=X1X2] {$|X_1|$};
  \node (X) [right of=X1] {$|X|$};
  \node (X2) [right of=X1X2] {$|X_2|$};
  \draw[->] (X1X2) to node {} (X2);
  \draw[->] (X1X2) to node [swap] {} (X1);
  \draw[->] (X1) to node [swap] {} (X);
  \draw[->] (X2) to node [swap] {} (X);
\end{tikzpicture}$$
    \noindent is also a pushout diagram. Choose neighborhoods $U_1,U_2$ of $X_1$ and $X_2$ in $X$ such that $U_1$ and $U_2$ retract to $X_1$ and $X_2$ in $X$ and $U_1\cap U_2$ retracts into $X_{12}$. The open cover $U_1,U_2$ of $X$ satisfies the conditions of \cite[Thm. 14.3.1]{BrownBook}, and thus the diagram:
\begin{equation}\label{eq:coequalizerUs}
\rho(|U_1\cap U_2|)\rightrightarrows \rho(|U_1|)\sqcup\rho(|U_2|)\to \rho(|X|) 
\end{equation}
is a pushout. The following diagram is thus also a pushout:
\begin{equation}\label{eq:coequalizerXs}
\rho(|X_{12}|)\rightrightarrows \rho(|X_1|)\sqcup\rho(|X_2|)\to \rho(|X|) 
\end{equation}
It is easily seen that the horizontalization functor $H:\omegagpd\to\inftygrpd$ is right adjoint. From this and from the diagram \ref{eq:coequalizerXs} we obtain a coequalizer diagram:
\begin{equation}
    \rho^\bigcirc(|X_{12}|)\rightrightarrows \rho^\bigcirc(|X_1|)\sqcup\rho^\bigcirc(|X_2|)\to \rho^\bigcirc(|X|) 
\end{equation}
Applying the presehaf $\mbox{Hom}_{\inftygrpd}(\_,\rho^\bigcirc(G)[-1])$ to the above diagram provides us with the pullback diagram
\[Hom_{\inftygrpd}(\rhogpdglobX,\rho^\bigcirc(G)[-1])\to\mathcal{M}^\bigcirc(X_1,G)\times \mathcal{M}^\bigcirc(X_2,G)\rightrightarrows\mathcal{M}^\bigcirc(Y,G)\]
The  concludes the proof of the claim. Doing induction on the argument above, we obtain that for any good cover $\{X_1,...,X_n\}$ of $X$, the equation $Hom_{\inftygrpd}(\rhogpdglobX,\rho^\bigcirc(G)[-1])=lim\mathcal{M}^\bigcirc(X_i,G)$. This, together with Observation \ref{obs:Global} concludes the proof of the theorem.

\end{proof}

\begin{obs}[Gauge transformations on global globular LHGFs]
    The previous theorem lets us think of global globular LHGFs as morphisms in $\mathcal{M}_{gl}^\bigcirc(X,G)$. Within this picture, a gauge transformation on global globular LHGFs is determined by a single function $u:X_0\to G_0$ defined on the discrete set $X_0$ instead of being determined by a collection of functions $u_i$ associated with the elements of the good cover. The  condition that the transition functions restricted to $X_0$ must evaluate to the constant function equal to $1$ synchronizes the local gauge transformations in the only points where they are relevant. 
\end{obs}

\begin{comment}
** NOTE TO JUAN: Without further explanation, the set 
$\mathcal{M}^\bigcirc(X,G)$ should not be associated with gauge fields because $X$ may not be contractible. 

In the theorem above $X_1, X_2$ are contractible and we may also consider only situations in which $Y$ is contractible. Then the theorem talks about local LHGFs which are compatible with respect to a simple compatibility condition (trans functions equal 1), but the result is not a set of gauge fields yet. We can make it into a gauge field by a definition or comment following the theorem (or the corollary). 
In my original version of this section the point of view written in the previous paragraph was a definition given before the theorem. The goal was to be able to say that the theorem solved the local to global problem. (I think that we can still say that the theorem and corollary solve the local to global problem. I like the presentation in the way it is now, but I also liked the previous one.) 
\end{comment}

\subsection{Coarse graining and a complex of global LHGFs}\label{ss:CoarseGrGlobalLHGF}

\begin{obs}\label{obs:ComplexOfGlobalLHGF}
In light of Theorem \ref{thm:localglobal}, and by Proposition \ref{prop:EnrichedInternalGroupoids}, the set of 
global globular LHGFs on $X$ with gauge internal group $K = \rho (G_\ast)$ naturally inherits the structure of an $\infty$-groupoid $[\rho^\bigcirc(|X|),\rho (G_\ast)[-1]]$ as the globular version of what is 
explicitly written in Observation \ref{obs:ClosedLHGF}. 
\noindent Considering the geometric realization $|[\rho^\bigcirc(|X|),\rho (G_\ast)[-1]]|$ of $[\rho^\bigcirc(|X|),\rho (G_\ast)[-1]]$ we obtain a convenient topological space. We have arguments suggesting that the homotopy type of this space could describe the homotopy type of the space of $G$ gauge fields over $|X|$. Recall that the set of connected componnents of the space of $G$-gauge fields over $|X|$ is in one to one correspondence with the set of equivalence classes of $G$-bundles over $|X|$, and that once bundle is selected the space of gauge fields on that bundle is an affine space. 

The objective of classifying $G$-bundles over $|X|$ by means of the algebraic study of the homotopy of $[\rho^\bigcirc(|X|),\rho (G_\ast)[-1]]$ is not finished. Completing it is one of our current research interests; this will require us to use the homotopy theory of diffeological spaces \cite{ChristensenWu} and a categorical view gauge fileds as presented in \cite{WaldorfSchreiber}. 
\end{obs} 

\begin{obs}\label{obs:coarsgrainingGlobalLHGFs}
This observation follows closely Observation \ref{obs:coarsgraining}. 
Given $X,Y\in\Simp$. If $Y$ refines $X$, then $\rho(|X|)\subseteq \rho(|Y|)$. This inclusion defines a morphism $j:\rho^\bigcirc(|X|)\to \rho^\bigcirc(|Y|)$ in $\Gpd_{\infty\Gpd}$. The pullback $j^\ast$ of $j$, defines a function $\mathcal{M}_{gl}^\bigcirc(Y,G) \to \mathcal{M}_{gl}^\bigcirc(X,G)$, expressing the coarse graining of global globular LHGFs on $Y$ to global globular LHGFs on $X$. 
\end{obs}

\section{Global globular LHGF and extended lattice gauge fields}\label{s:ELGF}
We remind the reader that in this paper the convention for composition of paths, and globes of paths, follows the natural writing order (from left to right) instead of convention traditionally used in differential geometry and physics. As a result, in our presentation of Extended Lattice Gauge Fields (ELGFs) given below, we follow the opposite composition convention as compared to the original definitions in \cite{MenesesZapata1,MenesesZapata2}. 

We relate LHGFs as defined in Section \ref{s:LHGF} to ELGFs as presented in \cite{MenesesZapata1,MenesesZapata2}. 
We prove that every global globular LHGF restricts to an ELGF. We do this by recasting the main components involved in the definition of ELGF in terms of the machinery developed in the previous sections. 
A combination of our results and the results in \cite{MenesesZapata1,MenesesZapata2} will allow us to associate a principal bundle to every global globular LHGF which is unique up to bundle equivalence.

\subsection{Path simplexes}\label{ss:SimplexesofPaths}
\noindent As in the rest of this paper, $X$ stands for a simplicial set. The definitions given below extends the notion of ELGF to any simplicial set; in its original version, they were introduced for the case in which $X$ is the barycentric subdivision of the simplicial complex associated to the triangulation of a manifold $M$, see \cite{MenesesZapata2}. Consider $M$ with a good (closed) cover associated to the simplices in $X$. Let $G$ be a Lie group with the trivial filtration. 
\begin{comment}
   We will assume, for the reminder of this section, that $X$ is the barycentric subdivision of the nerve of an open cover of a manifold $M$, see \cite{MenesesZapata2}. Consider $M$ with the good (closed) cover given by  simplices in $X$. Let $G$ be a Lie group with the trivial filtration.  
\end{comment}
An ELGF determines a lattice $G$-gauge field; i.e. a morphism $A^L: P^1 \rho (|X|_1) \to BG$, and complements it with homotopy data: 
Certain homotopy globes of paths, called $k$-simplexes of paths in \cite{MenesesZapata2}, are considered; an ELGF assigns to them $k$-homotopy globes in $G$. 
As a result of this extension, apart from $A^L$, an ELGF determines a principal $G$ bundle over $M$ up to bundle equivalence. 

We recast the definition of path simplex as singular cubes on $|X|$.

\begin{notation}\label{not:Sigma}
Let $m\geq 1$. We will denote by $\Sigma_m$ the map $[0,1]^m\to\Delta_m$ defined by the formula
\begin{equation}
\Sigma_m(a_1,...,a_m)=(a_1,...,\Pi_{i=1}^ma_i)
\end{equation}
for every $(a_1,...,a_m)\in [0,1]^m$. Geometrically:
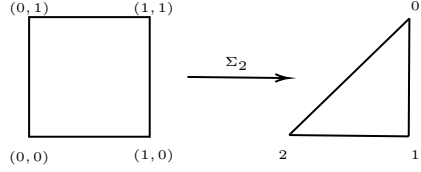
\begin{center}

\tikzset{every picture/.style={line width=0.75pt}} %set default line width to 0.75pt        

\begin{tikzpicture}[x=0.75pt,y=0.75pt,yscale=-1,xscale=1]

%Shape: Square [id:dp7699170103859152] 
\draw   (200.73,80.3) -- (260.9,80.3) -- (260.9,140.47) -- (200.73,140.47) -- cycle ;
%Straight Lines [id:da29398020906669586] 
\draw    (390.13,140.3) -- (330.47,139.63) ;
%Straight Lines [id:da5025050709944427] 
\draw    (390.13,140.3) -- (390.47,80.97) ;
%Straight Lines [id:da16266845271936936] 
\draw    (330.47,139.63) -- (390.47,80.97) ;
%Straight Lines [id:da22418767977164888] 
\draw    (279.73,110.4) -- (328.13,110.94) ;
\draw [shift={(330.13,110.97)}, rotate = 180.64] [color={rgb, 255:red, 0; green, 0; blue, 0 }  ][line width=0.75]    (6.56,-1.97) .. controls (4.17,-0.84) and (1.99,-0.18) .. (0,0) .. controls (1.99,0.18) and (4.17,0.84) .. (6.56,1.97)   ;

% Text Node
\draw (297.07,98.98) node [anchor=north west][inner sep=0.75pt]  [font=\tiny]  {$\Sigma _{2}$};
% Text Node
\draw (189.07,146.2) node [anchor=north west][inner sep=0.75pt]  [font=\tiny]  {$( 0,0)$};
% Text Node
\draw (323.73,145.87) node [anchor=north west][inner sep=0.75pt]  [font=\tiny]  {$2$};
% Text Node
\draw (251.2,145.05) node [anchor=north west][inner sep=0.75pt]  [font=\tiny]  {$( 1,0)$};
% Text Node
\draw (251.2,71.05) node [anchor=north west][inner sep=0.75pt]  [font=\tiny]  {$( 1,1)$};
% Text Node
\draw (189.2,71.05) node [anchor=north west][inner sep=0.75pt]  [font=\tiny]  {$( 0,1)$};
% Text Node
\draw (389.73,146.05) node [anchor=north west][inner sep=0.75pt]  [font=\tiny]  {$1$};
% Text Node
\draw (389.73,71.05) node [anchor=north west][inner sep=0.75pt]  [font=\tiny]  {$0$};

\end{tikzpicture}
\end{center}
when $m=2$. We orient $\Delta_m$ in such a way that $\Sigma_m(0,...,0)$ is the vertex $m$ in $\Delta_m$ for every $m$ as in the diagram above. The numbering on the rest of the vertices of $\Delta_m$ is done inductively following the above convention. Thus defined $\Sigma_m$ is a singular $m$-cube in $\Delta_m$. 
\end{notation}

\begin{obs}\label{obs:Sigmapartial}
Let $m\geq 1$. The singular cube $\Sigma_m$ in $\Delta_m$ is such that $\partial_0^0\Sigma_m$ is contained in the largest vertex $m$ of $\Delta_m$ and $\partial_0^1\Sigma_m$ is the face opposite to $m$
\end{obs}
\noindent In the following definition, and in the rest of this section we denote simplexes in a simplicial set $X$ by the letters $c_\tau,c_\nu$, etc. and we write every simplex of the form $c_\tau$ as being generated by its ordered vertex set. We do this so as to agree with the notational conventions used in \cite{MenesesZapata2}.
\begin{notation}\label{not:Gamma}
Let $X\in\Simp$. Let $c_\tau$ be an $m$-simplex in $X$. Write $c_\tau$ as $[v_0,...,v_m]$ where $\left\{v_0,...,v_m\right\}$ is the ordered vertex set of $c_\tau$. Let $\tilde{c_\tau}:\Delta_m\to X$ be the affine map such that $\tilde{c_\tau}(i)=v_i$ for every $i$. Let $\tilde{\Gamma}_{\tau}:[0,1]^m\to |X|$ denote the composition $\tilde{c_\tau}\Sigma_m$. Thus defined $\tilde{\Gamma}_{\tau}$ is a singular cube on $X$.
\end{notation}
\begin{notation}\label{not:Gammatilde}
    Let $X\in\Simp$ and $m\geq 0$. Let $c_\tau\in X^m$ and $c_\nu\in X^n$ be such that $c_\nu\subseteq c_\tau$. Write $\tau=[v_0,...,v_m]$ as in \ref{not:Gamma} and assume that $c_\nu$ is contained in the face opposite to $v_m$. Write $c_\nu=[v_{i_0},...,v_{i_n}]$ where $\left\{i_0,...,i_m\right\}\subseteq\left\{0,...,m\right\}$ is the vertex set of $c_\nu$. Let $\tilde{c}_{\nu\tau}$ be the affine map $\Delta_{n+1}\to X$ such that $\tilde{c}_{\nu\tau}(j)=i_j$ for every $j\leq n-1$ and $\tilde{c}_{\nu\tau}(n)=v_m$. Denote by $\tilde{\Gamma}_{\nu\tau}$ the composition $\tilde{c}_{\nu\tau}\Sigma_m$. Thus defined $\tilde{\Gamma}_{\nu\tau}$ is a singular $n+1$-cube on $X$ with image the subsimplex of $c_\tau$ generated by $c_\nu$ and the maximal vertex $v_m$ of $c_\tau$. 
\end{notation}
\begin{notation}
Let $X\in\Simp$ and $m\geq 0$. Let $c_\tau\in X^m$ and $c_\nu\in X^n$ be as in \ref{not:Gammatilde}. We will represent the homotopy cube $\tilde{\Gamma}_{\nu\tau}$ pictorially as:
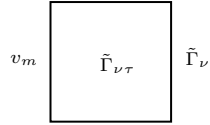
\begin{center}

\tikzset{every picture/.style={line width=0.75pt}} %set default line width to 0.75pt        

\begin{tikzpicture}[x=0.75pt,y=0.75pt,yscale=-1,xscale=1]

%Shape: Square [id:dp02638477394577099] 
\draw   (220.73,100.3) -- (280.9,100.3) -- (280.9,160.47) -- (220.73,160.47) -- cycle ;

% Text Node
\draw (243.92,124.91) node [anchor=north west][inner sep=0.75pt]  [font=\scriptsize]  {$\tilde{\Gamma}_{\nu \tau}$};
% Text Node
\draw (199.07,125.2) node [anchor=north west][inner sep=0.75pt]  [font=\scriptsize]  {$v_{m}$};
% Text Node
\draw (286.64,121.91) node [anchor=north west][inner sep=0.75pt]  [font=\scriptsize]  {$\tilde{\Gamma }_{\nu }$};

\end{tikzpicture}
\end{center}
By Observation \ref{obs:Sigmapartial}, if $v_m$ is the largest vertex of $c_\tau$ then the cube $\tilde{\Gamma}_{\nu\tau}$ satisfies the equations $\partial_0^0\tilde{\Gamma}_{\nu\tau}=v_m$ and $\partial_0^1\tilde{\Gamma}_{\nu\tau}=\tilde{\Gamma}_v$, thus the notation on the left and right edges of the above diagram. We write $\Gamma_{\nu\tau}$ for the composite cube $\tilde{\Gamma}_{\nu\tau}+_0\Gamma^0_{\tilde{\Gamma}_\nu}$, where $\Gamma^0_{\tilde{\Gamma}_\nu}$ is the connection cube on $\tilde{\Gamma}_\nu$. Pictorially $\Gamma_{\nu\tau}$ is the composite:
\begin{center}

\tikzset{every picture/.style={line width=0.75pt}} %set default line width to 0.75pt        

\begin{tikzpicture}[x=0.75pt,y=0.75pt,yscale=-1,xscale=1]

%Shape: Square [id:dp43987401131287984] 
\draw   (240.73,120.3) -- (300.9,120.3) -- (300.9,180.47) -- (240.73,180.47) -- cycle ;
%Shape: Square [id:dp030852990187631768] 
\draw   (300.73,120.3) -- (360.9,120.3) -- (360.9,180.47) -- (300.73,180.47) -- cycle ;
%Straight Lines [id:da9115580026600858] 
\draw    (310.22,130.02) -- (350.44,130.02) ;
%Straight Lines [id:da6208301519928816] 
\draw    (350.44,130.02) -- (350.67,170.13) ;

% Text Node
\draw (263.92,145.42) node [anchor=north west][inner sep=0.75pt]  [font=\scriptsize]  {$\tilde{\Gamma}_{\nu \tau} $};
% Text Node
\draw (322.92,145.42) node [anchor=north west][inner sep=0.75pt]  [font=\scriptsize]  {$\Gamma^0$};

\end{tikzpicture}
\end{center}
\end{notation}
\begin{obs}\label{obs:GammasGlob}
 Let $X\in\Simp$ and $m\geq 0$. Let $c_\tau\in X^m$ and $c_\nu\in X^n$ be as in \ref{not:Gammatilde}. The cube $\Gamma_{\nu\tau}$ is globular.   
\end{obs}
\begin{obs}[Path simplices]\label{obs:Ls}
Let $X\in\Simp$ and $m\geq 0$. Let $c_\tau\in X^m$ and $c_\nu\in X^n$ be as in \ref{not:Gammatilde}. Let $a\in [0,1]^{n-1}$. Write $L_a(t)$ for the line $t\mapsto (t,a)$, $t\in [0,1]$ in $[0,1]^n$. The composite $\Gamma_{\nu\tau}L_a(t)$ is a path starting at $v_m$ and ending at the largest vertex $v_{i_n}$ of $c_\nu$. If $i\in\left\{0,1\right\}$ and $b\in [0,1]^{n-2}$ we will write $L_b^i(t)$ for the line $t\mapsto (t,b,i)$, $t\in [0,1]$. The composites $\Gamma_{\nu\tau}L_b^i$ are paths starting at $v_m$ and ending at $v_{i_n}$ respectively. Pictorially:

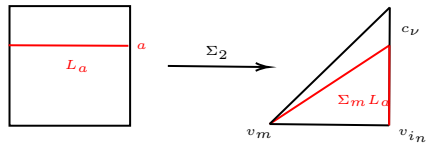
\begin{center}

\tikzset{every picture/.style={line width=0.75pt}} %set default line width to 0.75pt        

\begin{tikzpicture}[x=0.75pt,y=0.75pt,yscale=-1,xscale=1]

%Straight Lines [id:da18443240355481105] 
\draw [color={rgb, 255:red, 255; green, 0; blue, 0 }  ,draw opacity=1 ]   (370.27,80.04) -- (310.47,119.63) ;
%Shape: Square [id:dp5835670385542846] 
\draw   (180.73,60.3) -- (240.9,60.3) -- (240.9,120.47) -- (180.73,120.47) -- cycle ;
%Straight Lines [id:da3853597474650845] 
\draw    (370.13,120.3) -- (310.47,119.63) ;
%Straight Lines [id:da5375161452612462] 
\draw    (370.13,120.3) -- (370.47,60.97) ;
%Straight Lines [id:da41966838918012184] 
\draw    (310.47,119.63) -- (370.47,60.97) ;
%Straight Lines [id:da3857287711003996] 
\draw    (259.73,90.4) -- (308.13,90.94) ;
\draw [shift={(310.13,90.97)}, rotate = 180.64] [color={rgb, 255:red, 0; green, 0; blue, 0 }  ][line width=0.75]    (6.56,-1.97) .. controls (4.17,-0.84) and (1.99,-0.18) .. (0,0) .. controls (1.99,0.18) and (4.17,0.84) .. (6.56,1.97)   ;
%Straight Lines [id:da9951335200002189] 
\draw [color={rgb, 255:red, 255; green, 0; blue, 0 }  ,draw opacity=1 ]   (240.04,80.27) -- (180.29,79.99) ;
%Straight Lines [id:da47386144530561536] 
\draw [color={rgb, 255:red, 255; green, 0; blue, 0 }  ,draw opacity=1 ]   (370.27,80.04) -- (370.13,120.3) ;

% Text Node
\draw (277.07,78.98) node [anchor=north west][inner sep=0.75pt]  [font=\tiny]  {$\Sigma _{2}$};
% Text Node
\draw (242.62,77.42) node [anchor=north west][inner sep=0.75pt]  [font=\tiny,color={rgb, 255:red, 255; green, 0; blue, 0 }  ,opacity=1 ]  {$a$};
% Text Node
\draw (206.84,85.87) node [anchor=north west][inner sep=0.75pt]  [font=\tiny,color={rgb, 255:red, 255; green, 0; blue, 0 }  ,opacity=1 ]  {$L_{a}$};
% Text Node
\draw (342.66,103.3) node [anchor=north west][inner sep=0.75pt]  [font=\tiny,color={rgb, 255:red, 255; green, 0; blue, 0 }  ,opacity=1 ]  {$\Sigma _{m} L_{a}$};
% Text Node
\draw (297.07,120.98) node [anchor=north west][inner sep=0.75pt]  [font=\tiny]  {$v_{m}$};
% Text Node
\draw (373.07,120.98) node [anchor=north west][inner sep=0.75pt]  [font=\tiny]  {$v_{i_{n}}$};
% Text Node
\draw (374.5,68.98) node [anchor=north west][inner sep=0.75pt]  [font=\tiny]  {$c_{\nu }$};

\end{tikzpicture}
\end{center}
\end{obs}
\noindent In the notation of Observation \ref{obs:Ls} we think of the homotopy cubes $\Gamma_{\nu}$ as homotopies between paths of the form $\Gamma_{\nu\tau}L_b^0$ and $\Gamma_{\nu\tau}L_b^1$, $b\in [0,1]^{n-2}$ along paths of the form $\Gamma_{\nu\tau}L_a(t)$. This is precisely the definition of simplex of paths provided in \cite{MenesesZapata2} in the case in which $X$ is a triangulation of a smooth manifold $M$ obtained as the barycentric subdivision of the nerve of an open cover of $M$. Our notation is meant to reflect this identification. We will call the singular cubes $\Gamma_{\nu\tau}$ with $c_\tau$ and $c_\nu$ satisfying the conditions of \ref{not:Gammatilde} \textbf{simplexes of paths}. We warn the reader the the order of indices $\tau\nu$ in \ref{not:Gamma} and the order of the indices in the definition of simplexes of paths appearing in \cite{MenesesZapata2} is inverted. The following observation answers the question of when a simplex of paths restricts to another simplex of paths.
\begin{obs}\label{obs:restriction}
  Let $X\in\Simp$ and $m\geq 0$. Let $c_\tau\in X^m$ and $c_\nu\in X^n$ be as in \ref{not:Gammatilde}. Let $c_\varphi$ be a simplex in $c_\tau$ contained in the face opposite to $v_m$. If $c_\varphi\subseteq c_\nu$, then $\Gamma_{\varphi\tau}$ is a face of $\Gamma_{\nu\tau}$ precisely when $c_\varphi$ is contained in the face opposite to the vertex $v_{i_n}$. If now $c_\nu\subseteq c_\varphi$, then $\Gamma_{\nu\varphi}$ is a face of $\Gamma_{\nu\tau}$ precisely when $c_\nu$ is contained in the face opposite to the maximal vertex of $c_\varphi$.
\end{obs}

\subsection{LHGF and ELGF}

\noindent An extended lattice gauge field (ELGF) assigns relative homotopy classes of gluing map extensions to simplexes of paths of different dimensions, obeying a certain set of consistency conditions. We prove that every global globular LHGF restricts to an ELGF. The following is a version of the definition of ELGF appearing in \cite[Section 5.3]{MenesesZapata2}, adapted to the language used in the rest of the paper.

\begin{definition}\label{def:ELGF}
    Let $X\in\Simp$ and $G$ be a Lie group. An ELGF on $X$, with gauge group $G$, is an assignment $A:\Gamma_{\tau\nu}\mapsto \rho^\bigcirc(G)_k$ for every pair of simpleces $\overline{c_\nu}\subsetneq \overline{c}_\tau$ such that $c_\nu$ is of dimension $k$, satisfying the following compatibility conditions:
   \begin{enumerate}
      \item \label{item:ELGF1} Let $c_\tau,c_\nu,c_\sigma$ be simplexes in $X$. Suppose $\overline{c_\sigma}\subsetneq \overline{c_\nu}\subsetneq \overline{c}_\tau$. The following equation holds:
       \begin{equation}\label{eq:ELGFeq1}
           A(\Gamma_{\tau\sigma})=A(\Gamma_{\tau\nu}) \odot A(\Gamma_{\nu\sigma})
       \end{equation}
       \item \label{item:ELGF2}Let $c_\nu,c_\tau$ be simplexes in $X$. Suppose $c_\nu$ is a $k$-simplex. Consider evaluations $ A(\Gamma_{\tau\rho})$, with $c_\rho$ running through all simplices of $\partial c_\nu$. We assume that such a collection of globes assembles into the boundary of $A(\Gamma_{\tau \nu})$.     
   \end{enumerate}
\end{definition}
%Things to change: (2) change intro to subsection 

\begin{thm}\label{thm:ELGFvsLHGF10}
   Let $X\in\Simp$. Let $G$ be a Lie group. 
   The restriction of any $A\in \mathcal{M}^\bigcirc_{gl}(X,G)$ to simplexes of paths defines $A\in\mathcal{M}^{EL}(X,G)$. 
\end{thm}

\begin{proof}
   Given an inclusion of simplexes $\overline{c_\nu}\subseteq \overline{c}_\tau$ in $X$, the path simplex $\Gamma_{\tau\nu}$ is a $k+1$-dimensional singular globe in $\rho^\bigcirc(|X|)$, and thus $A(\Gamma_{\tau\nu})\in\rho^\bigcirc(G)_k$. The equation \ref{eq:ELGFeq1} follows directly from the fact that $A$ is a morphism of $\infty$-groupoids. The restriction of $A$ to simplexes of paths thus satisfies condition \ref{item:ELGF1} of Definition \ref{def:ELGF}. Condition \ref{item:ELGF2} follows from the compatibility of $A$ with boundaries, and thus with the operation $\partial$ of considering shells of globes. The manifold $A(\partial\Gamma_{\tau\nu})$ is the manifold assembled by the evaluations of the form $ A(\Gamma_{\tau\rho})$, with $c_\rho$ running through all simplexes of $c_\nu$, and is also the boundary of the contractible manifold $A(\Gamma_{\tau\nu})$. This concludes the proof.
\end{proof}

\begin{obs}
    The restriction of a globular global HLGF $A \in \mathcal{M}_{gl}^\bigcirc(X,G)$ to 1-globes 
    determines a lattice $G$-gauge field; i.e. a morphism $A^L: P^1 \rho (|X|_1) \to BG$. 
\end{obs}

Apart from containing the information of a usual lattice gauge field, the higher dimensional information in a HLGF determines a $G$-bundle over the base manifold. 

\begin{cor}\label{cor:paralleltransport}
Let $X\in\Simp$. Let $G$ be a Lie group. Let $A \in \mathcal{M}_{gl}^\bigcirc(X,G)$. 
$A$ determines a $G$-bundle $\pi$ over $|X|$ 
unique up to bundle equivalence. 
\end{cor}
%$A \in Hom_{\infty Gpd}(\rho^\bigcirc |X|, \rho^\bigcirc G [-1])$
\begin{proof}
In \cite{MenesesZapata2}  the analogous statement is proved for ELG fields on 
the barycentric subdivision of a simplicial complex triangulating a manifold. 
The proof of that statement is done by a construction which applies to any simplicial set 
$|X|$ with an ELG field $A^{EL}$. 
By Theorem \ref{thm:ELGFvsLHGF10} $A$ induces an ELG field $A^{EL}(A)$. This concludes the proof. 
\end{proof}

Notice, however, that given our choice of composition convention for paths and globes of paths 
the $G$-bundle constructed would be defined following the convention in which there is a global left action on the fibers, and parallel transport maps would be equivariant with respect to that action (instead of being equivariant with respect to a global right $G$ action on the fibers, as is traditionally required). (See explanation in the next section.)

\begin{obs}[Topology for the space of global globular LHGFs]
Theorem \ref{thm:ELGFvsLHGF10} gives us a function $Res: \mathcal{M}^\bigcirc_{gl}(X,G) \to \mathcal{M}^{EL}(X,G)$. In turn, $\mathcal{M}^{EL}(X,G)$ is a smooth manifold; it is a covering space of the space of lattice gauge fields $\mathcal{M}^{L}(X,G) \cong G^{N-1}$ the Lie group constructed as the product resulting from associating a copy of $G$ to each 1-simplex in $X_1$ \cite{MenesesZapata1}. 
We conjecture that the function $Res$ is a bijection. We provide $\mathcal{M}^\bigcirc_{gl}(X,G)$ with the topology cogenerated by the function $Res$. If the function $Res$ is indeed a bijection, then $\mathcal{M}^\bigcirc_{gl}(X,G)$ would itself become a smooth manifold, and moreover a covering space for $\mathcal{M}^{L}(X,G) \cong G^{N-1}$. One advantage of this topology is that the set of connected components of $\mathcal{M}^\bigcirc_{gl}(X,G)$ would be in one to one correspondence with the set of equivalence classes of $G$-bundles over $M= |X|$ \cite{MenesesZapata1}. 
\end{obs}

\section{Higher homotopy parallel transport}\label{s:parallel}
Again, we remind the reader that in this paper the composition convention for paths, and globes of paths, follows the natural writing order (from left to right) instead of convention traditionally used in differential geometry and physics. In regard to the parallel transport that our gauge fields induce, this change of convention means that the group element associated to a concatenation of paths is a product in the group with the factors corresponding to the sub paths ordered from left to right. This composition rule commutes with a global left $G$ action on the fibers. That is, the fiber bundles that considered here do not follow the traditional convention of differential geometry, but the opposite convention.

In this final section we study how LHGFs define higher dimensional parallel transport. Given $X\in\Simp$ and $\{X_1,...,X_n\}$ a good cover of $X$, it is easy to see how globular $\lo$s define parallel transport along cubes of paths contained in a single $X_i$. For general cubes of paths not contained in a single $X_i$ it is not as straight forward. We describe parallel transport induced by global globular LHGFs. 

Parallel transport along a homotopy globe of paths should correspond to a certain homotopy of parallel transport maps along paths. Let ${\cal F}$ be a torsor for a Lie group $G$, and assign a copy of it ${\cal F}_x$ to every $x \in X_0$. It is clear that $\rho^\bigcirc({\cal F}_x)$ is a torsor for $\rho^\bigcirc (G)$. 
Consider parallel transport along a globe of paths $\gamma \in \rho^\bigcirc (X)_k$. 
We will see that it is reasonable to identify initial conditions for that parallel transport as belonging to $\rho^\bigcirc({\cal F}_{s(\gamma)})_{k-1}$. 
A particular case with a very clear geometric interpretation is the parallel transport along $\gamma$ of a constant $(k-1)$-globe in 
$\rho^\bigcirc({\cal F}_{s(\gamma)})_{k-1}$; the geometric picture is that the result of the parallel transport is a $(k-1)$-globe corresponding to the parallel transport of the same initial condition along a $(k-1)$-globe homotopy of paths. We formalize this idea below.

\begin{definition}[Fiber homotopy data]
Given $x \in X_0$ 
\[
\phi_x \in \rho^\bigcirc({\cal F}_x)
\] 
will be called fiber homotopy data at $x$. A trivialization over $x$ establishes an identification 
\[
\rho^\bigcirc({\cal F}_x) \leftrightarrow \rho^\bigcirc (G). 
\]
Under a change of trivialization described by a gauge transformation $u$ 
fiber homotopy data $\phi_x \in \rho^\bigcirc (G)$ transform by right  multiplication 
\[
(\phi_x')_j = (\phi_x)_j \epsilon^j u^{-1}(x) . 
\]	
\end{definition}
\begin{comment}
(THIS NOTATION IS IF IN THE CONTEXT OF GLOBAL GLOBULAR LHGFS WE HAVE NOT DEFINED GAUGE TRANFS AT A SINGLE VERTEX. IF WE DID IT WE SHOULD CHANGE $u_i(x)$ TO $u_x$.) 
\end{comment}
\begin{definition}[Higher parallel transport along globes of paths]
Consider a globe $\gamma \in \rho^\bigcirc (X)$. 
A LHGF 
$A : \rho^\bigcirc (X) \to \rho^\bigcirc G[-1]$ induces a 
higher parallel transport map for fiber homotopy data 
\[
PT_A(\gamma) : 
\rho^\bigcirc({\cal F}_{s(\gamma)}) \to \rho^\bigcirc({\cal F}_{t(\gamma)}) , 
\]
which with the aid of a local trivialization is determined by 
\[
\phi_{t(\gamma)} = 
\phi_{s(\gamma)} A(\gamma) . 
\]	
\end{definition}

We finish this section, and the article, with a description of the basic properties of this higher homotopy parallel transport. 
\begin{obs}
The parallel transport map has the following properties:
\begin{enumerate}
    \item $PT_A(\gamma)$ is a $\rho^\bigcirc (G)$ equivariant morphism of infinity groupoids. 
    \item $PT_A(\gamma)$  is independent of the choice of  trivialization over vertices. 
    \item The properties of $A : \rho^\bigcirc (X) \to \rho^\bigcirc G[-1]$ as a morphism of infinity grupoids translate into corresponding properties of $PT_A$. 
    \item Given a choice of trivialization over vertices, $PT_A$ determines $A$. 
\end{enumerate}
\end{obs}

\begin{comment}
Properties of this notion of higher parallel transport: 

1.- $PT_A(\gamma)$ is a morphism of infinity groupoids which is equivariant with the right $\rho^\bigcirc G$ action on homotopy fiber data. THIS SHOULD BE A TRIVIAL CONSEQUENCE OF THE DEFINITION. 

Notice that the above statements implicitly say that the morphisms 
$\rho^\bigcirc{\cal F}_{s(\gamma)} \to \rho^\bigcirc{\cal F}$ 
which are equivariant with the right $\rho^\bigcirc G$ action are an infinity grupoid which can be identified with $\rho^\bigcirc G$ (acting by left multiplication on the trivialized homotopy fibers). I AM NOT SURE HOW EASY IS TO PROVE THIS. 

2.- Additionally, considering $PT_A$ as acting on the globes of paths, a groupoid internal to infinity groupoids, it is a morphism from $\rho^\bigcirc X$ to the morphisms 
$\rho^\bigcirc{\cal F}_{s(\gamma)} \to \rho^\bigcirc{\cal F}$ 
which are equivariant with the right $\rho^\bigcirc G$ action 
[on one degree less, the $[-1]$ on the definition of the LHGF $A$]
(which is identify-able with $\rho^\bigcirc G$ acting by left multiplication). THIS SHOULD ALSO BE DIRECT FORM THE DEFINITION AFTER THE MIDLE STEP OF THE PREVIOUS PARAGRAPH IS SETTLED.

\end{comment}

\end{document}